\documentclass[a4paper,10pt]{article}
\usepackage[a4paper]{geometry}
\usepackage{amsmath}
\usepackage{mathtools}
\usepackage{hyperref}
\usepackage{amsfonts}
\usepackage{amssymb,amsbsy,amsfonts,amscd}
\usepackage{amsthm}
\usepackage[T1]{fontenc}
\usepackage[english]{babel}
\usepackage{pst-node}
\usepackage{pst-coil}
\usepackage{pstricks}
\usepackage{graphicx}
\usepackage{pstricks-add}
\usepackage{url}
\usepackage{amsmath}
\usepackage{fancyhdr}
\numberwithin{equation}{section}
\newtheorem{definition}{Definition}[section]
\newtheorem{theo}[definition]{Theorem}

\newtheorem{lemma}[definition]{Lemma}

\newtheorem{remark}{Remark}

\newcommand{\CX}{\mathcal{X}}
\newcommand{\beq}{\begin{equation}}
\newcommand{\eeq}{\end{equation}}
\newcommand{\fer}[1]{(\ref{#1})}
\newcommand{\p}{\partial}
\newcommand{\f}{\frac}
\newcommand{\al}{\alpha}
\newcommand{\R}{\mathbb{R}}

\author{H\'el\`ene Leman\thanks{CMAP, Ecole Polytechnique, UMR 7641, route de
    Saclay, 91128 Palaiseau Cedex-France; E-mail: \texttt{helene.leman@polytechnique.edu}}, Sylvie M\'el\'eard\thanks{CMAP, Ecole Polytechnique, UMR 7641, route de
    Saclay, 91128 Palaiseau Cedex-France; E-mail: \texttt{sylvie.meleard@polytechnique.edu}}, Sepideh Mirrahimi\thanks{Institut de Math\'ematiques de Toulouse; UMR 5219, Universit\'e de Toulouse; CNRS, UPS IMT, F-31062 Toulouse Cedex 9, France; E-mail: \texttt{Sepideh.Mirrahimi@math.univ-toulouse.fr}}}
\title{Influence of a spatial structure on the long time behavior of a competitive Lotka-Volterra type system}

\begin{document}
\maketitle

\begin{abstract}
To describe population dynamics, it is crucial to take into account jointly evolution mechanisms and spatial motion. However, the models which include these both aspects, are not still well-understood. Can we extend the existing results on type structured populations, to  models of populations structured by type and space, considering diffusion and nonlocal competition between individuals?\\
We study a nonlocal competitive Lotka-Volterra type system, describing a spatially structured population which can be either monomorphic or dimorphic. Considering spatial diffusion, intrinsic death and birth rates, together with death rates due to intraspecific and  interspecific competition between the individuals, leading to some integral terms, we analyze the long time behavior of the solutions.  We first prove existence of steady states and next  determine the long time limits, depending on the competition rates and the principal eigenvalues of some operators, corresponding somehow to the strength of  traits. Numerical computations illustrate that the introduction of a new mutant population can lead to the long time evolution of the spatial niche. 

\end{abstract}

\let\thefootnote\relax

\footnotetext{{\bf AMS Class. No:} {35Q92, 45K05, 35K57, 35B40}}

\section*{Introduction}

The spatial aspect of populations is an important ecological issue  which has been extensively studied (see \cite{DLM00}, \cite{DL94b},  \cite{McG99}, \cite{Murray89}, \cite{TK96}). The interplay between space and evolution is particularly crucial in the emergence of polymorphism and spatial patterns  and  the heterogeneity of the environment is considered as essential  (\cite{FM88}, \cite{K02}). The combination of spatial motion and mutation-selection processes is also known for a long time to have important effects on population dynamics (\cite{Endler77}, \cite{Mayr63}).
Recently biological studies observed that classical models could underestimate the invasion speed  and suggested that invasion and evolution are closely related. The ecological parameters can have a strong effect on the expansion of invading species and conversely, the evolution can be conditioned by the spatial behavior of individuals related to the resources available. 
The paper by Philipps and Co \cite{Pal06} shows the strong impact of the morphologic parameters of the cane toads on the expansion of their invasion.\\ 

\noindent In this context, the study of space-related traits, such as dispersal speed or sensibility to heterogeneously distributed resources, is fundamental and has been object of mathematical developments. In Champagnat-M\'el\'eard \cite{CM07}, a stochastic individual-based model is  introduced where individuals are characterized both by their location and one or several phenotypic and heritable traits. The individuals move, reproduce with possible mutation and die of natural death or because of competition for resources. The spatial motion is modeled as a diffusion and the  spatial interaction between individuals is modeled by a convolution kernel in some spatial range. In a large population scale, it is shown that this microscopic stochastic model can be approximated by a nonlinear nonlocal reaction-diffusion equation defined on the space of traits and space. The latter has been studied in Ferri\`ere-Desvillettes-Pr\'evost \cite{DFP04} and Arnold-Desvillettes-Pr\'evost \cite{ADP12} and existence and uniqueness of the solution, numerical simulations and steady states are studied. Propagation phenomena and existence of traveling waves are explored numerically and theoretically for different variants of such models in \cite{MA.JC.GR:13}, \cite{OB.VC.NM.RV:12}, \cite{HB.GC:12}, \cite{BC13}. This problem has also been studied from an asymptotic point of view using Hamilton-Jacobi equations   \cite{EB.VC.NM.SM:12}, \cite{EB.SM:13}. \\

\noindent Despite several recent attempts to study such models, dynamics of populations structured by trait and space are not completely understood and several interesting and challenging questions remain to be resolved in this field (see for instance \cite{EB.VC.NM.SM:12}). In particular, the works quoted above concentrate on the case where the mutations are frequent such that the diffusion in space and the mutations are modeled in the same time scale. Our objective is to understand the framework of adaptive dynamics where the mutations are rare enough such that between two mutations the dynamics is driven by a system of nonlocal reaction-diffusion equations, each of them describing the dynamics and the spatial distribution of one trait. We study the steady states and the long time behavior of such systems.  Note that although the existence of steady states for a model with continuous trait and space is provided in \cite{ADP12}, the long time behavior of solutions is not known, to our knowledge, for discrete or continuous traits. However, in the case of a single trait and considering only homogeneous environments, \cite{BNPR09} provides a study  of steady states and traveling waves.\\

\medskip \noindent In this paper we focus on this problem, for the simplest cases where the population is either monomorphic (a single type is involved) or dimorphic (the population is  composed of two-type subpopulations). The space state is  an open bounded subset $\mathcal{X}$ of $\mathbb{R}^d$ with a boundary of class $C^3$. In the dimorphic case, the spatial density of the population is modeled by the system of nonlinear partial differential equations of parabolic type
\begin{equation}
\label{eq:intro}
\left\{
\begin{aligned}
     \left\{
	\begin{aligned}
	&\partial_t g_1(t,x)= m_1 \Delta_x g_1(t,x)+ \bigg(a_1(x)-\int_{\mathcal{X}}I_{11}(y)g_{1}(t,y)dy  -\int_{\mathcal{X}}I_{12}(y)g_{2}(t,y)dy \bigg) g_1(t,x),\\
	&\partial_n g_1(t,x)=0, \text{ on } \partial \mathcal{X}, 
	\end{aligned}
     \right.\\
     \left\{
	\begin{aligned}
	&\partial_t g_2(t,x)=m_2 \Delta_x g_2(t,x)+ \bigg(a_2(x)-\int_{\mathcal{X}}I_{21}(y)g_{1}(t,y)dy)-\int_{\mathcal{X}}I_{22}(y)g_{2}(t,y)dy \bigg) g_2(t,x),\\
	&\partial_n g_2(t,x)=0, \text{ on } \partial \mathcal{X}, 
	\end{aligned}
     \right.
\end{aligned}
\right.
\end{equation}
where $g_{1}(t,x)$ (respectively $g_{2}(t,x)$) denotes the density of individuals of type $1$ (resp. of type $2$), in position $x$ at time $t$. 
 The  density dynamics is driven by  growth rates $a_{1}$ and $a_{2}$ which depend on the spatial position of individuals and on their type. Further the competition is modeled by  nonlocal death rates depending on the environment heterogeneity through  kernels $I_{ij}, i,j = 1,2$. We will show that this system admits $4$ non-negative steady states  depending on the ecological parameters. The stability of these  states is based on the signs of the principal eigenvalues 
 \begin{equation*}
\begin{aligned}
&H_1=-\min_{{\underset{u\not\equiv 0}{u \in H^1}}}  \dfrac{1}{\|u\|^2_{L^2}} \left[ \int_{\mathcal{X}} m_1 |\nabla u|^2dx-\int_{\mathcal{X}}a_1(x)u(x)^2dx\right],\\
&H_2=-\min_{{\underset{u\not\equiv 0}{u \in H^1}}}  \dfrac{1}{\|u\|^2_{L^2}} \left[ \int_{\mathcal{X}} m_2 |\nabla u|^2dx-\int_{\mathcal{X}}a_2(x)u(x)^2dx\right].\\
\end{aligned}
\end{equation*}
Here $H^1$ is the Sobolev space of order $1$ on $\mathcal{X} $.
The first steady state is the trivial one $(0,0)$ describing an extinct population, two of them describe  long term specialization on a single type and only one subpopulation has a non trivial long time behavior. The last one describes a co-existence case where individuals with two types exist in a long time scale.  \\
 
 \noindent The main results of the paper (Theorems \ref {theo:main} and \ref{theo:special})
give assumptions based on spectral parameters and competitive kernels under which the solution of the equation converges, as time goes to infinity, to one of these steady states.
 This result is new and interesting by itself but it will also be the first step in an adaptive dynamics framework, if we want to understand how mutant individuals invade the population at the evolutive scale (see \cite{C06}, \cite{CFM08}).

\medskip
\noindent 
Section \ref{sec:mon} deals with the case of a monomorphic population where all individuals have the same type. The density dynamics is led by a nonlinear partial differential equation of parabolic type with a non-local competition term. In this case, we show existence of steady states for a more general competition term:
\begin{equation}
\label{eq:intro2}
 \left\{ \begin{aligned}
         &\partial_t g(t,x)= m\Delta g(t,x)+ a(x)g(t,x)- \left(\int_{\mathcal{X}}I(x,y)g(t,y)dy\right) g(t,x),  \;  \forall (t,x) \in \mathbb{R}^+ \times \mathcal{X}\\
	 &\partial_n g(t,x)=0, \; \forall (t,x) \in \mathbb{R}^+ \times \partial \mathcal{X},\\
	 &g(0,x)=g^0(x).
        \end{aligned}
\right.
\end{equation}
 We then explore the long-time behavior of the solution, for the particular case $I(x,y)=I(y)$. Note that, the long-time behavior of the solution of \fer{eq:intro2} with general competition kernel $I(x,y)$ is not yet understood to our knowledge. In  \cite{Co13} the steady states and the long time behavior of the solution of a similar model are studied using different techniques,  for the monomorphic case. However, we provide a shorter result for the convergence in long time of solutions which is easily generalizable to dimension 2.\\
 
 \noindent Section \ref{sec:dim} is devoted to the two-type case of  dimorphic population. In this section, we present our main results on the steady states and the long-time behavior of the  solution for  the system of partial differential equations \eqref{eq:intro}. We also present some numerical results which confirm that the introduction of a new mutant population can lead to a coexistence or an invasion of the previous population, that is linked with an evolution of spatial niches.\\

 \noindent Finally in Section \ref{sec:proof} we provide the proofs of our main Theorems \ref{theo:main} and \ref{theo:special} on the long time behavior and the stability of steady states for a dimorphic population.\\

\noindent The mathematical analysis rely on the spectral decomposition of compact operators, fixed point arguments and the study of perturbed Lotka-Volterra type systems.\\

\noindent {Notation:}

\noindent The space set $\mathcal{X}$ is an open bounded subset of $\mathbb{R}^d$ with a boundary of class $C^3$. We will denote by ${L}^k$ the Lebesgue space on $\mathcal{X}$  of order $k\in \mathbb{N}^*$ and by ${H}^k$ the Sobolev space on $\mathcal{X}$  of order $k \in \mathbb{N}^*$. We denote by $C^{0,1}$ the space of Lipschitz continuous functions on $\mathcal{X}$.\\
For all $x \in \partial \mathcal{X}$, we denote by $n(x)$ the outward normal to the boundary $\partial \mathcal{X}$ at point $x$. For a sufficiently smooth function $u$ and $x \in \partial\mathcal{X}$, we denote by $\partial_n u(x)$ the scalar product $\nabla u(x) . n(x)$.\\

\section{Monomorphic population}
\label{sec:mon}

In this section, we study the case of a monomorphic population. 
The model is written as below
\begin{equation}
\label{eq:parabolique}
 \left\{ \begin{aligned}
         &\partial_t g(t,x)= m\Delta g(t,x)+ a(x)g(t,x)- \left(\int_{\mathcal{X}}I(x,y)g(t,y)dy\right) g(t,x),  \;  \forall (t,x) \in \mathbb{R}^+ \times \mathcal{X}\\
	 &\partial_n g(t,x)=0, \; \forall (t,x) \in \mathbb{R}^+ \times \partial \mathcal{X},\\
	 &g(0,x)=g^0(x), \; \forall x \in \mathcal{X}.
        \end{aligned}
\right.
\end{equation}
Here $g(t,x)$ denotes the density of individuals in position $x$ and at time $t$. The Laplace term corresponds to the diffusion of individuals in space and the positive constant $m$ is the rate of this diffusion. We represent the intrinsic growth rate by $a$, which depends on the position of individuals. Finally, the last term corresponds to the mortality due to competition. We denote indeed by $I(x,y)$ the competition rate between individuals at position $x$ and individuals at position $y$.\\

We first give a necessary and sufficient condition for \fer{eq:parabolique} to have a steady solution. We then prove, in a particular case where $I(x,y)=I(y)$, that the solution converges in long time to the unique steady solution of \fer{eq:parabolique}.

\subsection{Existence of steady state}

In this section, we give a necessary and sufficient condition for \fer{eq:parabolique} to have a steady solution. To this end, we make the following assumptions on the coefficients:
\beq
\label{as:a}
a\in C^{0,1}(\CX),\quad \text{and} \quad |a(x)|\leq a_\infty, \quad \text{for all $x\in \mathcal{X}$,}
\eeq
\begin{equation}
\label{hyp:c}
\left\{
\begin{aligned}
&I(\cdot,\cdot)\in C(\bar{\CX}\times \bar{\CX}) \text{ is nonnegative and Lipschitz continuous with respect to the first variable,}\\
&\text{if $d=1$ : } \exists I_- >0 / \; \forall x\in \mathcal{X}, \; I(x,x)\geq I_-,  \\
&\text{if $d>1$ : } \exists I_->0 / \; \forall (x,y) \in \mathcal{X}\times \mathcal{X}, \;  I(x,y) \geq I_- ,
\end{aligned}
\right.
\end{equation}
Consequently, there exists a positive constant $I^+$ such that
\beq
\label{as:K2}
I(x,y)\leq I^{+},\quad \text{for all $x,y\in\mathcal{X}$. }
\eeq
Here, we note that these assumptions are not necessarily optimal. In particular, the regularity of the coefficients may be relaxed but this is sufficient for our purpose and allows us to avoid some technical details. \\

\noindent To state our result we also need the following lemma which can be derived easily from Krein-Rutman's Theorem (see for instance the chapter 6 of \cite{Evans}) and its proof is left to the reader. 

\begin{lemma}[{\bf Eigenvalue problem}]\label{lem:ev}
There exists a principal eigenvalue $H$ to the following eigenvalue problem:
\begin{equation}
 \label{eq:eigenfunction}
\begin{cases}
m\Delta \overline u(x)+a(x)\overline u= H \overline{u},   &   \forall x\in \mathcal{X},\\
\partial_n \overline u(x)=0,   &   \forall  x\in \p\mathcal{X}.
\end{cases}
\end{equation}
This eigenvalue is simple and the corresponding eigenfunction $\overline u$ is the only eigenfunction which is strictly positive in  $\mathcal{X}$. Moreover, $H$ can be computed from the following variational problem
\begin{equation*}
 H=-\min_{{\underset{u\not\equiv 0}{u \in H^1}}} \dfrac{1}{\|u\|^2_{L^2}} \left[ \int_{\mathcal{X}} m|\nabla u|^2dx-\int_{\mathcal{X}}a(x)u^2(x)dx\right].
\end{equation*}
\end{lemma}

\noindent We are now ready to state our first result:
 
\begin{theo}[{\bf Existence of steady state}]
\label{theo:stationnaire}
Assume \fer{as:a}, \fer{hyp:c}. (i) If $H\leq 0$, then there is no non-trivial nonnegative steady solution for  \fer{eq:parabolique}. (ii) If $H>0$, then \fer{eq:parabolique} has a strictly positive steady solution  $\bar{g}\in C^{2}(\CX)$, i.e. $\bar g$ solves
\begin{equation}
\label{eq:stationnaire}
\left\{
\begin{aligned}
&-m\Delta \bar{g}(x)=\left(a(x)-\int_{\mathcal{X}} I(x,y)\bar{g}(y)dy \right)\bar{g}(x), \; \forall x \in \mathcal{X}\\
&\partial_n \bar{g}(x)=0, \; \forall x \in \partial \mathcal{X}.
\end{aligned}
\right.
\end{equation}
\end{theo}

\bigskip

\begin{proof}

(i) Let $H\leq 0$. We prove by contradiction that there is no nonnegative solution to \fer{eq:stationnaire}. To this end, we suppose that $0\leq \overline g\in C^{2}(\CX)$ solves \fer{eq:stationnaire}. Supposing that $\overline g$ is non-trivial, from the maximum principle we obtain that $\overline g$ is strictly positive and in particular
$$
\int I(x,y)\overline g(y)\overline g(x)\overline u(x) dxdy>0.
$$
We now multiply \fer{eq:stationnaire} by $\overline u$ and integrate with respect to $x$ to obtain from \fer{eq:eigenfunction},
$$
H\int \overline u \overline g dx = \int I(x,y)\overline g(y)\overline g(x)\overline u(x) dxdy>0.
$$
This is in contradiction with the assumption $H\leq 0$.\\

(ii) We now suppose that $H>0$. To prove that \fer{theo:stationnaire} has a steady solution, we construct a mapping 
\begin{equation*}
   \Upsilon : \left( \begin{aligned}
                      L^2 &&& \rightarrow && L^2\\
                      h &&& \mapsto && g
                     \end{aligned}
	      \right), 
\end{equation*}
 such that any fixed point of this mapping will be a steady state of our problem, as follows.\\
\noindent Thanks to  \fer{as:a}, we can choose $\delta >0$  small enough such that $1-\delta a(x) >0$ for all $x \in \CX$.  Let 
$h\in L^2$. We define $\psi(h)=h\left(1- \delta \int I(\cdot,y)h(y)dy\right)$, and   $\Upsilon(h)=g$, where $g\in H^1(\CX)$ is the unique solution of the following equation  
 \beq
\label{map}
 \begin{cases}
 -m\delta \Delta g(x)-\delta a(x) g(x)+g(x)=\psi(h)(x),&\text{in $\CX$,}\\
 \p_n g(x)=0,& \text{on $\p\CX$.}
\end{cases}
\eeq
Notice that fixed points of the mapping $\Upsilon$ are steady solutions of our problem and conversely. So the last step is to show that $\Upsilon$ has a fixed point. We establish this result thanks to Schauder's fixed point Theorem (see for instance Theorem (4.1) in \cite{cronin}).\\
We first notice from the choice of $\delta$ that  $(-m\delta \Delta + (1-\delta a )Id)^{-1}$ is a continuous and compact mapping. As $\psi:L^2\to L^2$ is a well-defined continuous mapping, we deduce the continuity and compactness of $\Upsilon$.\\

\noindent
We will split the rest of the proof into two cases  depending on the dimension of the domain, denoted by $d$, as in the statement of the theorem.\\

$\bullet$ If $d>1$, using Lemma \ref{lem:ev}, there exists a positive eigenfunction $\overline u$ associated to the positive eigenvalue $H$. We denoted by $\overline{u}^+$ and $\overline{u}_-$ its maximum and minimum values on $\CX$. Then we define
\begin{equation}
\label{def:lambda}
 \lambda^+=\frac{H \overline{u}^+}{I_-} \quad \text{  and } \quad \lambda_-=\frac{H\overline{u}_-}{I^+}
\end{equation}
 and choose $\delta>0$ small enough such that 
\begin{equation}
\label{as:delta}
 \lambda^+\leq \frac{\overline{u}_-}{2 \delta I^+}, \quad \text{ and } \quad 1-\delta H>0.
\end{equation}
\noindent
Let us now introduce the convex closed subset of $L^2$
\begin{equation*}
 \mathcal{Y}=\left\{g\in L^2 | g\geq 0, \lambda_- \leq \int_{\CX} g \overline u \leq \lambda^+   \right\}.
\end{equation*}
\noindent
We now prove that $\Upsilon$ maps $\mathcal{Y}$ into itself.\\
Let $h$ be in $\mathcal{Y}$, and $g=\Upsilon(h)$, they satisfy 
\begin{equation}
\label{eq:gh}
 -m \delta \Delta g(x)-\delta a(x) g(x) +g(x) = h(x) \left( 1- \delta \int_{\CX} I(x,y)h(y)dy \right), \quad \text{in $\CX$}.
\end{equation}
As $h$ is a positive function, and using \eqref{as:delta},
\begin{equation*}
 \psi(h) \geq h\left(1- \delta\int_{\CX} \frac{I(.,y)}{\overline u (y)} h(y) \overline u(y)dy \right) \geq h(1- \delta \frac{I^+}{\overline u_-} \lambda^+) \geq \frac{h}{2} \geq 0.
\end{equation*}
We deduce that $g$ is positive on $\CX$ thanks to the maximum principle.\\
Then we multiply \eqref{eq:gh} by $\overline{u}$ and integrate it over $\CX$,
\begin{equation*}
 \int_{\CX} (-m \delta \Delta g-\delta a g)\overline{u} + \int_{\CX}g \overline u =  \int_{\CX} h\overline{u}- \delta \int_{\CX\times \CX} \frac{I(x,y)}{\overline u(y)}h(y)\overline{u}(y)h(x)\overline{u}(x)dydx
\end{equation*}
From an integration by parts, \eqref{eq:eigenfunction}, \eqref{as:K2} and \eqref{hyp:c}, we find the following inequality:
\begin{equation*}
 \int_{\CX} h\overline{u} \left( 1- \delta \frac{I_-}{\overline u ^+} \int_{\CX} h\overline{u} \right) \geq (1-\delta H) \int_{\CX}g \overline u \geq \int_{\CX} h\overline{u} \left( 1- \delta \frac{I^+}{\overline u_-} \int_{\CX} h\overline{u} \right).
\end{equation*}
Thanks to \eqref{as:delta}, the two polynomial functions $r \mapsto r ( 1- \delta \frac{I_-}{\overline u ^+} r )$ and $r \mapsto r ( 1- \delta \frac{I^+}{\overline u_-} r )$ are increasing on interval $[\lambda_-, \lambda^+]$, that is, as $\int_{\CX} h\overline{u} \in [\lambda_-, \lambda^+]$,
\begin{equation*}
 \lambda^+ \left( 1- \delta \frac{I_-}{\overline u^+} \lambda^+ \right) \geq (1-\delta H) \int_{\CX}g \overline u \geq \lambda_- \left( 1- \delta \frac{I^+}{\overline u_-} \lambda_- \right).
\end{equation*}
Finally we obtain from \eqref{def:lambda} and \eqref{as:delta} that $\lambda^+ \geq  \int_{\CX}g \overline u \geq \lambda_-$, that is $g\in \mathcal{Y}$.\\

\noindent We conclude from the Schauder's fixed point theorem that $\Upsilon$ has a positive fixed point.\\

$\bullet$ For $d=1$, the previous proof is valid if $I$ is strictly positive in $\CX$ but we can relax this assumption to the one in \eqref{hyp:c} thanks to the following method. 
We first prove the following lemma
\begin{lemma}
 \label{lemme:upsilon}
 Assume \fer{as:a} and \fer{hyp:c}. There exists $R>0$ such that for all $g \in L^2$, positive and $t \in [0,1[$ satisfying $g=t \Upsilon(g)$, we have $\|g\|_{L^2} < R$.
 \end{lemma}

\begin{proof}
We will use an argument which is similar to the one presented in \cite{BNPR09}. Let $g\in L^2$, positive and $t\in ]0,1[$ such that $g=t\Upsilon(g)$. $g$ attain its maximum value at a point $x_0 \in \overline{\mathcal{X}}$. As $g$ satisfies Neumann boundary conditions, for all $x_0 \in \overline{\mathcal{X}}$, we have $ g'(x_0)=0 \; \text{ and } \; g''(x_0)\leq 0$. Using \fer{map} at the point $x_0$ and since $t<1$, we get
\begin{equation}
\label{eq:maja}
\int_{\mathcal{X}} I(x_0,y)g(y)dy \leq a_{\infty}.
\end{equation}
We then use Taylor-Lagrange's formula for the function $g$ at point $x_0$. For all $y \in \overline{\mathcal{X}}$, there exists $\xi \in ]x_0,y[$ or $]y,x_0[$ such that 
$$
 g(y)=\|g\|_{\infty}+{(y-x_0)^2}g''(\xi)/{2}.
$$ 
Additionally, by using again \fer{map}, we obtain, for all $\xi \in \overline{\mathcal{X}}$, $g''(\xi) \geq  -( \|g\|_{\infty} a_{\infty})/m$. We deduce that $g(y) \geq \|g\|_{\infty} \left( 1-a_{\infty}\dfrac{(y-x_0)^2}{2m} \right)_+$. Therefore \eqref{eq:maja} implies
\begin{equation*}
\|g\|_{\infty} \leq a_{\infty} \left( \int_{\mathcal{X}} I(x_0,y) \left( 1-a_{\infty}\dfrac{(y-x_0)^2}{2m} \right)_+dy \right)^{-1}<+\infty,
\end{equation*}
which is bounded since $I(x_0,.)$ is positive in a neighborhood of $x_0$ from \fer{hyp:c}, we conclude easily.
\end{proof}

\noindent Thanks to this lemma, we choose $\delta$ satisfying
\begin{equation}
 \label{as:delta2}
\delta < \min \left( \frac{\overline{u}_-}{2 R I^+ \|\overline u\|_{L^2}}, \frac{1}{I^+R \sqrt{|\CX|} } \right).
\end{equation}
Then we define the convex closed subset 
$$
\mathcal{Y}=\{h\in L^2 | h\geq 0, \|h\|_{L^2}\leq R, \int_{\CX} h \overline u \geq \lambda_- \},
$$ 
where $\lambda_-$ is defined as before by \eqref{def:lambda}.
For $h\in \mathcal{Y}$, we have $\int_{\CX}I(x,y)h(y)dy \leq I^+  R\sqrt{|\CX|} < \frac{1}{\delta}$ which implies that $\psi(h)$  and $g=\Upsilon(h)$ are positive functions. Moreover, following similar arguments as in the case $d>1$, and noticing that the assumption \eqref{as:delta2} guarantees that $\int_{\CX} h \overline u \,dx \in [\lambda_-, \frac{\overline{u}_-}{2 \delta I^+} ]$, we obtain that 
$\int_{\CX} g \overline u \geq \lambda_-.$\\
As we are not sure that $\|\Upsilon(h)\|_{L^2}\leq R$, we use the following method inspired by the proof of Schaefer's fixed point theorem: we introduce a new function $\tilde{\Upsilon}: L^2 \rightarrow L^2$,
\begin{equation*}
\tilde{\Upsilon}(h) = \begin{cases}
  \Upsilon(h), & \text{ if } \| \Upsilon(h)\|_{L^2} \leq R\\
  \frac{R}{\|\Upsilon(h)\|_{L^2}} \cdot \Upsilon(h), & \text{ if }  \|\Upsilon(h)\|_{L^2} > R.
 \end{cases}
\end{equation*}
Thus $\tilde{\Upsilon}$ is a continuous and compact function that maps $\mathcal{Y}$ into itself. From the Schauder's fixed point theorem, we deduce that it has a fixed point $\bar{g}$. Suppose that $\| \Upsilon(\bar{g}) \|_{L^2}> R$, then $\tilde{\Upsilon}(\bar{g})=\frac{R}{\| \Upsilon(\bar{g})\|_{L^2}} \Upsilon(\bar{g})=\bar{g}$ 
and $\|\tilde \Upsilon(\bar{g})\|_{L^2}=\|\bar{g}\|_{L^2}=R$. However, since $\bar g=t\Upsilon(\bar{g})$ with $t=\f{R}{\|\Upsilon(\bar{g})\|_{L^2}}<1$,  this is in  contradiction with
Lemma \ref{lemme:upsilon}. Hence $\|\Upsilon(\bar{g})\|_{L^2}\leq R$ and $\bar{g}$ is a fixed point of $\Upsilon$.\\

\end{proof}

\subsection{Convergence to steady state solution}
\label{ssec:convmono}
We now come back to our initial problem and we study the long time behavior of the solution to \eqref{eq:parabolique} for a particular case, where $I$ satisfies the following assumptions:
\begin{equation}
\label{as:par}
 I\in L^2(\mathcal{X})\quad \text{and there exists $I_->0$, for all $y \in \mathcal{X}$, $I(y) \geq I_-$.,}
\end{equation}
We also assume the following assumption on the initial condition
\beq
\label{as:ini}
g^0\in L^2(\CX).
\eeq

\noindent We show that the positive solution of the following parabolic equation 
\begin{equation}
\label{eq:champmoyen}
 \left\{ \begin{aligned}
         &\partial_t g(t,x)= m\Delta_x g(t,x)+ a(x)g(t,x)-\left(\int_{\mathcal{X}}I(y)g(t,y) dy\right) g(t,x),  \;  \forall x \in \mathcal{X}\\
	 &\partial_n g(t,x)=0, \; \forall x \in \partial \mathcal{X},  \; \forall t \in \mathbb{R},\\
	 & g(0,x)=g^0(x), \;  \forall x \in \mathcal{X}.
        \end{aligned}
\right.
\end{equation}
tends to the unique steady state of the problem while $t \to +\infty$.
\begin{theo}
\label{theo:convmono}
Assume \fer{as:a}, \fer{as:par} and \fer{as:ini}. If $H>0$, any positive $C^2$-solution to \eqref{eq:champmoyen} tends in $L^{\infty}$ to the unique positive solution to
\begin{equation}
 \label{eq:stationnairecm}
 \left\{ \begin{aligned}
         &- m\Delta \bar{g}(x)= a(x)\bar{g}(x)- \left(\int_{\mathcal{X}}I(y)\bar{g}(y)dy\right) \bar{g}(x),  \;  \forall x \in \mathcal{X}\\
	 &\partial_n \bar{g}(x)=0, \; \forall x \in \partial \mathcal{X},  \; \forall t \in \mathbb{R}. 
        \end{aligned}
\right.
\end{equation}
Moreover, if $H\leq0$, $g(t,\cdot) \overset{L^{\infty}}{\underset{t \rightarrow +\infty}{\longrightarrow}} 0$.\\
\end{theo}

\begin{remark}Notice that $H>0$ is a necessary and sufficient condition to obtain a positive limit as $t\to \infty$, that is, only the diffusion parameter and the growth rate $a$ have an influence on the non-extinction of the population in long time. The competition rate affects only the total size of the population at the limit, but not its existence. \\
We also note that a simple example where the assumption $H>0$ is satisfied is the case where the growth rate $a$ is a positive function, that is, if the birth rate is greater than the death rate everywhere, the population will survive. Moreover, if $a$ is a negative function, it is easy to deduce that $H$ is negative and that the population goes extinct.\\

\end{remark}

\begin{proof}

We first check that there exists only one positive steady state in the case $H>0$. Let $\bar{g},\bar{h} \in H^1$ be two positive solutions to \eqref{eq:stationnairecm}. Hence $\bar{g}, \bar{h}$ are positive eigenvectors of the compact, continuous operator $\mathcal{L}=m\Delta(.)+a.$. As Lemma \ref{lem:ev} implies the uniqueness of a positive eigenvector up to a multiplicative constant, $\bar{g}=s \cdot \bar{h}$ with $s\in \mathbb{R}^+$. Moreover, from \eqref{eq:stationnairecm}, we deduce that the principal eigenvalue $H$ is equal to $ \left(\int_{\mathcal{X}}I(y)\bar{g}(y)dy\right)$ and the same result holds for $\bar{h}$. It follows that $\bar{g}=\bar{h}$. We now denote by $\bar{g}$ the unique solution of \eqref{eq:stationnairecm}.\\

The next step is to show the convergence in $L^{\infty}$ towards the positive steady state if $H>0$ and towards $0$ if $H\leq 0$. Let us make the following changes of variable function
\begin{equation*}
 \forall (t,x) \in \mathbb{R} \times \mathcal{X}, \quad v(t,x)=g(t,x) \exp\left(\int_{0}^t \Big(\int_{\mathcal{X}} I(y)g(s,y)dy \Big) ds \right).
\end{equation*}
Thus $v$ is a solution to the equation
\begin{equation}
\label{eq:v}
\left\{
\begin{aligned}
& \partial_t v(t,x) -  m\Delta v(t,x)=a(x) v(t,x), & \forall (t,x) \in \mathbb{R} \times \mathcal{X},\\
& \partial_n v(t,x)=0, & \forall (t,x) \in \mathbb{R} \times \partial\mathcal{X},\\
& v(0,x)=g^0(x),& \forall x \in  \partial\mathcal{X}.
\end{aligned}
\right.
\end{equation}
It is well-known from the spectral decomposition of the operator $\mathcal{L}$ and the regularizing property of the Laplace operator that $v(t,\cdot) e^{-Ht}$ tends uniformly to $\beta \bar g$, a principal eigenvector of the operator $\mathcal{L}$, that is, for some positive constant $\beta$,
\begin{equation}
\label{limit:KAM}
 g(t,x) \exp\left( \int_0^t \Big(\int_{\mathcal{X}} I(y) g(s,y)dy \Big) ds-Ht\right) \overset{L^{\infty}}{\underset{t \rightarrow +\infty}{\longrightarrow}} \beta\bar{g}.
\end{equation}
 We divide this limit by an integrated version of it to obtain
\begin{equation}
\label{eq:limiteg/ro}
 \dfrac{g(t,\cdot)}{\rho(t) } \overset{L^{\infty}}{\underset{t \rightarrow + \infty}{\longrightarrow}} \dfrac{\bar g(\cdot)}{\int_{\mathcal{X}} \bar g dx}>0,
\end{equation}
where $\rho(t)=\int_{\mathcal{X}} g(t,y) dy$, and the r.h.s. is positive since $\bar g$ is a principal eigenvector of $\mathcal{L}$.\\
It remains to show that $\rho(t)$ has a finite limit when $t$ tends to infinity. Integrating \eqref{eq:champmoyen}, we find that $\rho$ is a solution to
\begin{equation*}
\f{d}{dt} \rho(t)= \left( \int_{\mathcal{X}}a(y)\dfrac{g(t,y) }{\rho(t)}dy - \int_{\mathcal{X}}I(y)\dfrac{g(t,y) }{\rho(t)}dy \cdot \rho(t) \right) \rho(t), \; \forall t\in \mathbb{R}.
\end{equation*}
Moreover $-m\Delta \bar g = a\bar g - H\bar g$ using definitions of $H$ and $\bar g$, which leads to $\int_{\mathcal{X}}a(x)\bar g(x)dx=H\int_{\mathcal{X}}\bar g(x)dx$. Therefore \eqref{eq:limiteg/ro} implies
\begin{equation*}
 \int_{\mathcal{X}}a(y)\dfrac{g(t,y) }{\rho(t)}dy \underset{t \rightarrow +\infty}{\longrightarrow} \int_{\mathcal{X}}a(y)\dfrac{\bar g(y)}{\int_{\mathcal{X}} \bar g}dy=H
 \end{equation*}
 and 
 \begin{equation*}\int_{\mathcal{X}}I(y)\dfrac{g(t,y) }{\rho(t)}dx \underset{t \rightarrow +\infty}{\longrightarrow} \mu:=\int_{\mathcal{X}}\dfrac{I(y)\bar g(y)}{\int_{\mathcal{X}} \bar g} dy.
\end{equation*}
 Thus, $\rho$ is solution to the equation $\f{d}{dt} \rho(t)=\rho(t) (H +\mathcal{D}(t)-\mu \rho(t))$, where
\begin{equation*}
   \mathcal{D}(t)= \int_{\mathcal{X}} a(y)\dfrac{g(t,y) }{\rho(t)}dy-H +\left( \int_{\mathcal{X}} I(y)\dfrac{g(t,y) }{\rho(t)}dy-\mu \right) \rho(t) \underset{t \rightarrow +\infty}{\longrightarrow} 0
\end{equation*}
Indeed, since $I$ is positive, we have $\f{d}{dt} \rho(t) \leq (a_{\infty} -I_- \cdot \rho(t)) \rho(t)$, that is sufficient to conclude that $\sup_{t \in \mathbb{R}^+} \rho(t) <+\infty$, and $\mathcal{D}(t) \longrightarrow 0$, as $t \rightarrow +\infty$.\\
Next, we show that $\rho(t)$ tends to $H/\mu$ if $H>0$ and to $0$ if $H \leq 0$ thanks to the following lemma, proven below.

\begin{lemma}
\label{lemme:rho}
 Let $\mu \in \mathbb{R}^+_*$, and $\rho$ be a positive solution on $\mathbb{R}$ to $\f{d}{dt} \rho(t)=\rho(t) (r +\mathcal{E}(t)-\mu \rho(t))$, where $\mathcal{E}(t) \underset{t \rightarrow +\infty}{\longrightarrow} 0$, then $\rho(t)$ tends to $\dfrac{r}{\mu}$ if $r\geq0$ or to $0$ if $r<0$ as $t$ tends to $+\infty$.
\end{lemma}
\noindent
Finally, we conclude from \eqref{eq:limiteg/ro} and the above lemma that, for $H> 0$,
\begin{equation*}
 {g(t,\cdot)} \overset{L^{\infty}}{\underset{t \rightarrow + \infty}{\longrightarrow}} \dfrac{H}{\mu}  \dfrac{\bar g}{\int_{\mathcal{X}}\bar g\, dx}=\bar{g}, \quad\text{ since } \quad \int_{\mathcal{X}}\bar{g} \, dx=\dfrac{H}{\mu},
 \end{equation*}
 and for $H\leq0$, 
\begin{equation*}
 {g(t,\cdot)} \overset{L^{\infty}}{\underset{t \rightarrow + \infty}{\longrightarrow}}0.
  \end{equation*}

\end{proof}


\begin{proof}[Proof of Lemma \ref{lemme:rho}]
We split this proof into two parts, depending on the value of $r$.
\begin{itemize}
 \item If $r<0$, there exists $t_0 \in \mathbb{R}^+$ such that for all $t\leq t_0$, $r+\mathcal{E}(t)-\mu \rho(t)<\frac{r}{2}$, i.e. $\partial_t \rho(t) \leq \frac{r}{2} \rho(t)$ and we conclude with Gronwall's lemma.
\item If $r\geq 0$, fix $\varepsilon_0 >0$. there exists $t_0$ such that for all $t \geq t_0$, $|\mathcal{E}(t)|\leq \varepsilon_0$, that is
	\begin{equation*}
	 \rho(t) \left( r -\varepsilon_0 -\mu \rho(t) \right) \leq \partial_t \rho(t) \leq \rho(t) \left( r +\varepsilon_0-\mu \rho(t) \right).
	\end{equation*}
\noindent
That means
	\begin{equation*}
	 \dfrac{r-\varepsilon_0}{\mu} \leq \underset{t \rightarrow +\infty}{\liminf} \rho(t) \leq \underset{t \rightarrow +\infty}{\limsup} \rho(t) \leq \dfrac{r+\varepsilon_0}{\mu}.
	\end{equation*}
\noindent
As this is true for all $\varepsilon_0>0$, we can conclude.
\end{itemize}
\end{proof}


\section{Dimorphic population}
\label{sec:dim}

In this section we study the dynamics of a system describing a dimorphic population. The model is written
\begin{equation}
 \label{eq:systeme}
\left\{
\begin{aligned}
     \left\{
	\begin{aligned}
	&\partial_t g_1(t,x)= m_1 \Delta_x g_1(t,x)+ \left(a_1(x)-\int_{\mathcal{X}}I_{11}(y)g_1(t,y)dy  -\int_{\mathcal{X}}I_{12}(y)g_2(t,y)dy \right) g_1(t,x),\\
	&\partial_n g_1(t,x)=0, \; \forall (t,x) \in \R^+\times \partial \mathcal{X}, \\
	&  g_1(0,x)=g_1^0(x) \; \forall x \in \mathcal{X},
	\end{aligned}
     \right.\\
     \left\{
	\begin{aligned}
	&\partial_t g_2(t,x)=m_2 \Delta_x g_2(t,x)+ \left( a_2(x)-\int_{\mathcal{X}}I_{21}(y)g_1(t,y)dy-\int_{\mathcal{X}}I_{22}(y)g_2(t,y)dy \right) g_2(t,x),\\
	&\partial_n g_2(t,x)=0,\; \forall (t,x) \in \R^+\times \partial \mathcal{X},\\
	&  g_2(0,x)=g_2^0(x) \; \forall x \in \mathcal{X}.
	\end{aligned}
     \right.
\end{aligned}
\right.
\end{equation}
Here, $g_1(t,x)$ (respectively $g_2(t,x)$) denotes the density of individuals of type $1$ (respectively of type $2$) at position $x\in \CX$ and at time $t\geq 0$. As before, the diffusion of individuals in space is modeled by Laplace terms and the diffusion rates for the populations of type $1$ and $2$ are denoted respectively by  the positive constants $m_1$ and $m_2$. We represent the intrinsic growth rate of individuals of type $i$, for $i\in \{1,2\}$, by $a_i$. The last terms correspond to the mortality due to intraspecific and  interspecific competition. The functions 
$I_{ij}(y)$, for $i,j\in \{1,2\}$, represent indeed the pressure applied by individuals of type $j$ at position $y$ on individuals of type $i$ at any position.

\noindent We first identify the steady solutions of \fer{eq:systeme}. Next, we study the  stability of steady states and the long time behavior of the solutions. To this end, we will make the following assumptions on the coefficients, for $i,j\in \{1,2\}$,
\beq
\label{as:ai}
a_i\in C^{0,1}(\CX),\quad \text{and} \quad |a_i(x)|\leq a_\infty, \quad \text{for all $x\in \mathcal{X}$,}
\eeq
\begin{equation}
\label{hyp:ci}
\left\{
\begin{aligned}
&\|I_{ij}\|_{L^2}< +\infty,\quad \text{for $i,j\in \{1,2\}$,}\\
& \exists I_- >0 / \; \forall x\in \mathcal{X}, \; I_{ii}(x)\geq I_-,\quad \text{for $i=1,2$}.
\end{aligned}
\right.
\end{equation}
We also make the following assumptions on the initial conditions
\beq
\label{as:di-ini}
g_i^0\in L^2(\CX), \quad \text{for $i=1,2$}.
\eeq

\noindent
Similarly to the case of a monomorphic population, the long time behavior depends on the values of spectral parameters. To this end we use  the spectral decomposition of compact operators:
\begin{lemma}[Spectral decomposition of compact operators (see chapter VI.4 of \cite{Brezis})]
\label{spec}
For $i\in \{1,2\}$, there exists a spectral basis $(\lambda_k^i,A_k^i)_{k\geq 1}$, for the operator $\mathcal{L}^i(u)=m_i\Delta u +a_i \cdot u$   with Neumann boundary condition, that is, \\
(i) $\lambda_k^i$ is a nondecreasing sequence with $H_i:=\lambda_1^i >\lambda_2^i\geq \lambda_3^i\geq \cdots \geq \lambda_k^i\geq\cdots$ and $\lambda^i_k\to -\infty$ as $t\to \infty$.\\
(ii) $(\lambda_k^i, A_k^i)$ are eigenpairs, that is for all $k\geq 1$ and $i=1,2$,
\begin{equation}
 \label{eq:ev2}
\begin{cases}
m_i\Delta  A_k^i(x)+a_i(x)A_k^i= \lambda_k^i A_k^i,   &   \forall x\in \mathcal{X},\\
\partial_n  A_k^i(x)=0,   &   \forall  x\in \p\mathcal{X}.
\end{cases}
\end{equation}
(iii) $(A_k^i)_{k\geq 1}$ is an orthogonal basis of $L^2(\CX)$.We normalize them by 
$$
\int_{\mathcal{X}} |A_k^i(x) | dx=1.
$$
(iv) The first eigenvalue $H_i$ is simple and is given by
$$
H_i=-\min_{{\underset{u\not\equiv 0}{u \in H^1}}}  \dfrac{1}{\|u\|^2_{L^2}} \left[ \int_{\mathcal{X}} m_i |\nabla u|^2dx-\int_{\mathcal{X}}a_i(x)u(x)^2dx\right].
$$
The first eigenfunction $A^i_1$ is positive, the eigenfunctions corresponding to the other eigenvalues change sign. Those eigenfunctions are smooth.
\end{lemma}  

\noindent Let us also introduce the following notation for $i,j\in \{1,2\}$
\begin{equation}
\label{def:mu}
 \mu_{ij}=\int_{\mathcal{X}} I_{ij}(x) {A_1^j}(x) dx.
 \end{equation}

\noindent From \eqref{hyp:ci}, we notice that $\mu_{11}\neq 0$ and $\mu_{22}\neq 0$. Moreover, let us remark that, in the particular case where the interaction kernels are constant, i.e.  $I_{ij} \equiv\bar I_{ij}$ (homogeneous environment), then $\mu_{ij}= \bar I_{ij}$.\\
Finally, we make the following assumption on the variables $\mu_{ij}$ to exclude a degenerate case
\beq
\label{as:det}
\mu_{11}\mu_{22}-\mu_{12}\mu_{21}\neq 0.
\eeq

\subsection{Studies of the steady states}

We first identify the steady states of the equation \eqref{eq:systeme}.
\begin{lemma}
\label{lemme:etatstationnaire}
Assume \fer{as:ai}, \fer{hyp:ci} and \fer{as:det}. Then, the only non-negative steady states of the equation \eqref{eq:systeme} are
   \begin{itemize}
\item the trivial steady state $(0,0)$,
\item the state $(\bar{g}_1,0)$  with  $\bar{g}_1=\f{H_1}{\mu_{11}}A_1^1$, which is non-negative and non-trivial if and only if $H_1>0$, 
\item the state $(0,\bar{g}_2)$ with $\bar{g}_2=\f{H_2}{\mu_{22}}A_1^2$, which is non-negative and non-trivial if and only if $H_2>0$, 
\item and the state $(\hat{g}_1, \hat{g}_2)$ where $\hat{g}_1=r_1A_1^1$ and $\hat{g}_2=r_2A_1^2$, with $r_1$ and $r_2$ positive constants satisfying
\begin{equation}
\label{H1H2}
 \left( \begin{array}{c}
 H_1\\
 H_2
 \end{array}
 \right)= \left( \begin{array}{cc}
 \mu_{11}& \mu_{12}\\
 \mu_{21}&\mu_{22}
  \end{array}
 \right)\left( \begin{array}{c}
 r_1\\
 r_2
 \end{array}
 \right).
\end{equation}
This steady state exists if and only if $H_1>0$, $H_2>0$ and $(H_2 \mu_{11}-H_1 \mu_{21})(H_1 \mu_{22}-H_2 \mu_{12})>0$.
\end{itemize}
\end{lemma}

\medskip
\begin{proof}
The conditions on the existence of the three first steady states are immediate from Theorem \ref{theo:stationnaire}. Moreover, it follows from Lemma \ref{spec} that there is no nonnegative steady state other than the ones stated above. We only prove the last point corresponding to the steady state with two positive exponents.\\
We first suppose that $r_1$ and $r_2$ given by \fer{H1H2} are positive. It is then easy to verify, from \fer{eq:ev2} and \fer{def:mu}, that $(r_1A_1^1,r_2A_1^2)$ is a steady solution of \fer{eq:systeme}.\\
We next notice using \fer{as:det} that the matrix in \fer{H1H2} is invertible, and $r_1$ and $r_2$ are positive if and only if 
\begin{equation*}
 r_i=\dfrac{H_i \mu_{jj}-H_j \mu_{ij}}{\mu_{jj}\mu_{ii}-\mu_{ji}\mu_{ij}}>0, \text{ for } (i,j)\in\{(1,2),(2,1)\}.
\end{equation*}
This is equivalent to $(H_2 \mu_{11}-H_1 \mu_{21})(H_1 \mu_{22}-H_2 \mu_{12})>0$ and $H_1,H_2>0$: indeed, if $H_2 \mu_{11}-H_1 \mu_{21}$ and $H_1 \mu_{22}-H_2 \mu_{12}$ have the same sign, then
\begin{equation*}
\mu_{11}\mu_{22}-\mu_{12}\mu_{21}=\dfrac{\mu_{11}}{H_1} \left( H_1 \mu_{22}-H_2 \mu_{12}\right)+\dfrac{\mu_{12}}{H_1} \left( H_2 \mu_{11}-H_1 \mu_{21}\right),
\end{equation*}
has also the same sign if and only if $H_1>0$. We conclude easily.

\end{proof}

\subsection{Long time behavior of the  system}

Our main results concern the long time behavior of solutions of \eqref{eq:systeme}. We give explicit conditions determining whether or not the population goes extinct or whether or not there is co-existence of the two types at equilibrium. The first theorem shows the convergence of the solution when time goes to infinity and gives sufficient conditions for convergence  to the  globally asymptotically stable states. The second theorem explores the more delicate cases, where there are several stable equilibria and different basins of attraction. 

\begin{theo}
\label{theo:main}
Assume \fer{as:ai}, \fer{hyp:ci}, \fer{as:di-ini} and \fer{as:det}.
 \begin{enumerate}
  \item \label{item:convergence} For any initial condition, as $t\to \infty$, the unique solution of the parabolic system  \eqref{eq:systeme} tends to one of the steady states described in Lemma \ref{lemme:etatstationnaire}.
  \item \label{item:00} If $H_1 \leq 0$ and $H_2 \leq 0$ then for any initial condition and as $t\to \infty$, the solution of \eqref{eq:systeme} tends to $(0,0)$, i.e. the population goes extinct. 

  \item \label{item:g10gas} If
	\begin{equation*}
	 H_1 > 0 , \;  H_2\mu_{11}-H_1\mu_{21} \leq 0 \text{ and } H_1\mu_{22}-H_2\mu_{12} >0,
	\end{equation*}
	then for any initial condition such that $g_1^0$ is not identically zero and as $t\to \infty$, the solution converges to $(\bar{g}_1,0)$. We thus have fixation of type 1 in the population.
	
  \item \label{item:0g2}  If
	\begin{equation*}
	 H_2 > 0 , \;  H_2\mu_{11}-H_1\mu_{21} > 0 \text{ and } H_1\mu_{22}-H_2\mu_{12} \leq0,
	\end{equation*}
	then for any initial condition such that $g_2^0$ is not identically zero and as $t\to \infty$, the solution converges to  $(0, \bar{g}_2)$. We thus have fixation of type 2 in the population.

  \item \label{item:coexistence} If
	\begin{equation*}
	 H_1 >0, \; H_2 >0, \; H_2\mu_{11}-H_1\mu_{21} > 0 \text{ and } H_1\mu_{22}-H_2\mu_{12} >0,
	\end{equation*}
	then for any initial condition such that $g_1^0$ and $g_2^0$ are not identically zero and as $t\to \infty$, the  solution converges to  $(\hat{g}_1, \hat{g}_2)$, i.e. we have co-existence of types $1$ and $2$.

 \end{enumerate}

\end{theo}

\bigskip
\noindent Before writing the second theorem which deals with the other values of $H_1$ and $H_2$, we illustrate this first theorem with some numerical examples. The numerics are computed with an algorithm based on finite difference method. Our aim is to illustrate the behavior of mutant individuals that appear in a well established population.\\
Fisrt of all, thanks to Theorem \ref{theo:convmono}, we have a mean to compute the principal eigenvalues $H_i$ when they are positive. In fact, in this case, Theorem \ref{theo:convmono} guarantees that any positive solution to $\partial_t u=m_i\Delta_x(u)+\big(a_i-\int_{\mathcal{X}}I_{ii}u\big) u$ with Neumann boundary conditions tends to the steady state $\bar{g}_i$, and 
\begin{equation}
\label{findH}
H_i=\int_{\mathcal{X}}I_{ii}(x) \bar{g}_i(x)dx.
\end{equation}
\noindent Thus, with the finite difference method, we resolve numerically the previous parabolic equation. After a long time, the solution is stable, so we consider that it has reached the steady state. We calculate then $H_i$ thanks to the simple formula \eqref{findH}.
With same ideas, we can also calculate $\mu_{i,j}$ for $i,j\in \{1,2\}$. Thus we can check the conditions of Theorem \ref{theo:main} for the following numerical examples, the values are presented on figure \ref{fig:simu}.\\

\noindent Let us now describe our numerical simulations. We consider that the growth rates of the two populations are maximal at two different spatial positions. For instance, the space state can represent a variation of resources, as seed size for some birds, and so the two populations are not best-adapted to same resources. 
Different values of $\bar{a}_2$, the maximum of the growth rate of the mutant population, will be explored, while the other parameters are fixed,
\beq
\label{data1}
\begin{array}{ccc}
&&\mathcal{X}=[0,1], \quad u_1=0.3, \: \bar{a}_1=1, \quad u_2=0.5,\\
&& a_i(x)= \max \{ \bar{a}_i(1-20(x-u_i)^2), -1\}.
\end{array}
\eeq
Notice that around the trait $u_i$, the growth rate of the population $i$ is positive but far from this position, it becomes negative. Thus positions around $u_i$ are favorable for population $i$, and we suppose that the intraspecific competition is greater around that position:
\beq
\label{data2}
I_{ii}(x)= \begin{cases}
            1, \quad \text{if } |x-u_i|<0.25,\\
	    0.1, \quad \text{else}.
           \end{cases}
\eeq
Then, we define the interspecific competition from the previous kernels by $I_{12}=I_{21}=\min \{ I_{11}, I_{22} \}$. Finally, we suppose that all individuals move with the same diffusion constant $m_1=m_2=0.01$.\\
As we want to illustrate the invasion of a mutant, we suppose that the initial condition is near $(\bar{g_1},0)$, as presented in figure \ref{fig:simu}(a). We resolve numerically the system of parabolic equations \eqref{eq:systeme} and present the solution after a long time, that is, when the densities are almost stable, see figure \ref{fig:simu}.
\begin{figure}[!h]
\begin{center}
\includegraphics[scale=0.33]{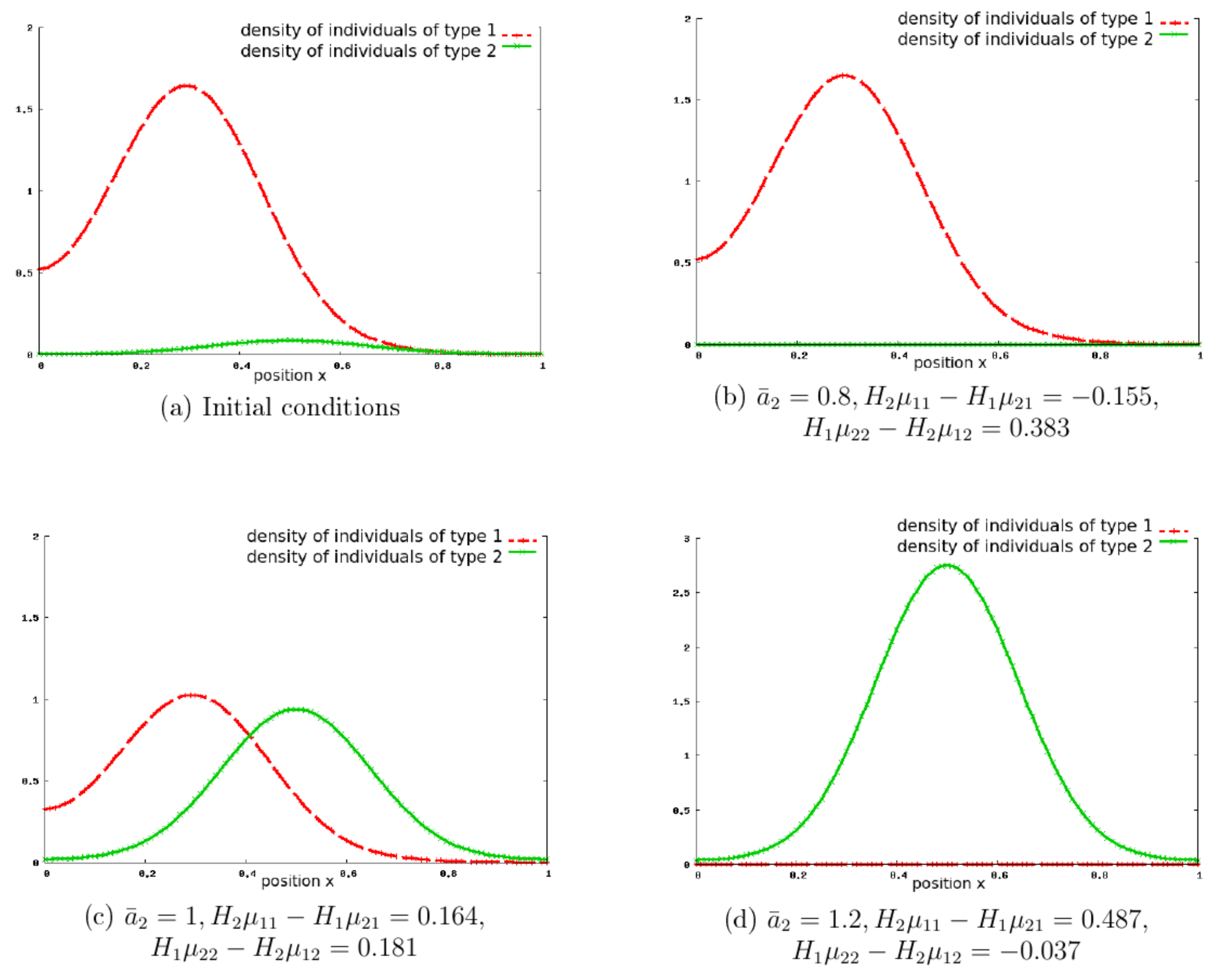}
 \caption{\label{fig:simu} \small{ The numerical resolution of \fer{eq:systeme} with parameters given by \fer{data1}--\fer{data2}. (a) presents the densities of each population initially. (b), (c) and (d) present the densities at time $t=1000$ for different values of $\bar{a}_2$, the red dashed curves represent the density of the resident population and the green ones the density of the mutants.}}
\end{center}
\end{figure}
\noindent
When $\bar{a}_2$ is small, the mutant population is not able to survive (case (b)). But when $\bar{a}_2$ is big enough, coexistence (case (c)) and even invasion (case (d)) can appear. On the two last cases, the new population that has invaded the space does not live on the same spatial position as the previous one. From an ecological viewpoint, such examples are very interesting because we observe a change of spatial niche due to a selection event.\\

\bigskip \noindent
Finally, to give a complete picture of the long time behavior of the solution, let us now study the last cases where several equilibria can be reached.

\begin{theo}
\label{theo:special}
Assume \fer{as:ai}, \fer{hyp:ci}, \fer{as:di-ini} and \fer{as:det}.
 \begin{enumerate}
  \item \label{item:special1} If 
	\begin{equation*}
	 \begin{aligned}
	  &H_1 > 0\;, \; H_2 >0, \; H_2\mu_{11}-H_1\mu_{21} <0,  \text{ and } H_1\mu_{22}-H_2\mu_{12}<0,
	 \end{aligned}
	\end{equation*}
	then the  steady states $(\bar{g}_1,0)$ and  $(0, \bar{g}_2)$ are both asymptotically stable and  $(\hat{g}_1, \hat{g}_2)$ is unstable. Nevertheless some solutions will converge to the latter. 
	
	 \item \label{item:special2} If 
	\begin{equation*}
	 \begin{aligned}
	  &H_1 > 0\;, \; H_2 >0, \; H_2\mu_{11}-H_1\mu_{21} <0, \text{ and }  H_1\mu_{22}-H_2\mu_{12}=0,
	 \end{aligned}
	\end{equation*}
	then the steady state $(\bar{g}_1,0)$ is asymptotically stable    and $(0, \bar{g}_2)$ is unstable. Nevertheless some solutions will converge to the latter. 
	
	 \item \label{item:special3} If 
	\begin{equation*}
	 \begin{aligned}
	  &H_1 > 0\;, \; H_2 >0, \; H_2\mu_{11}-H_1\mu_{21} =0, \text{ and }  H_1\mu_{22}-H_2\mu_{12}<0,
	 \end{aligned}
	\end{equation*}
	then the steady state $(\bar{g}_2,0)$ is asymptotically stable and $(0, \bar{g}_1)$ is unstable. Nevertheless some solutions will converge to the latter. 
	\end{enumerate}
\end{theo}

\begin{remark}
One can verify that, excluding a degenerate case by \fer{as:det}, all the possible values of $(H_1,H_2)\in \R^2$  are covered by the statements of Theorems \ref{theo:main} and \ref{theo:special}.
\end{remark}

\section{Stability of the steady states -The proofs of Theorems \ref{theo:main} and \ref{theo:special}}
\label{sec:proof}

\bigskip \noindent
This section is devoted to the proofs of Theorems \ref{theo:main} and \ref{theo:special}.

\bigskip \noindent
\subsection{ The proof of Theorem \ref{theo:main}}
In this section, we prove Theorem \ref{theo:main}. To this end,  noticing that the total density of the population is not constant, we will first study the limit of population densities normalized by the masses. Then we will study the long time behavior of a system of differential equations which describes the dynamics of the two masses (see Lemma \ref{lemme:rho1rho2}).

\bigskip \noindent
Similarly to the proof of Theorem \eqref{theo:convmono}, we make the following change of variables, for $i\in \{1,2\}$,
\begin{equation*}
 v_i(t,x)=g_i(t,x)\exp\left(\int_0^t\left(  \int_{\mathcal{X}}I_{i1}(y) g_1(s,y)dy+ \int_{\mathcal{X}} I_{i2}(y)g_2(s,y)dy \right)ds\right).
\end{equation*}
Following similar arguments as in subsection \ref{ssec:convmono}, we find a similar limit as \eqref{limit:KAM} for $g_i$, which leads to
\begin{equation}
\label{limit:gi/intgi}
\dfrac{g_1(t,.)}{\int_{\mathcal{X}}g_1(t,y)dy} \overset{L^{\infty}}{\underset{t \rightarrow + \infty}{\longrightarrow}} {A_1^1} \quad \text{ and } \quad \dfrac{g_2(t,.)}{\int_{\mathcal{X}}g_2(t,y)dy} \overset{L^{\infty}}{\underset{t \rightarrow + \infty}{\longrightarrow}} {A_1^2}. 
\end{equation}
Let $\rho_i(t)=\int_{\mathcal{X}}g_i(t,y)dy$ for $i\in \{1,2\}$. It remains now to understand the behavior of $(\rho_1(t),\rho_2(t))$. We deduce the following limits from \eqref{limit:gi/intgi},
\begin{equation*}
 \int_{\mathcal{X}}\dfrac{a_i(y)g_i(t,y)}{\rho_i(t)}dy \underset{t \rightarrow + \infty}{\longrightarrow} H_i \; \text{ and } \; \int_{\mathcal{X}}\dfrac{I_{ji}(y)g_i(t,y)}{\rho_i(t)}dy \underset{t \rightarrow + \infty}{\longrightarrow} \mu_{ji}.
\end{equation*}
Integrating \eqref{eq:systeme} on $\mathcal{X}$ and using the previous limits, we find that $(\rho_1, \rho_2)$ is a solution to
\begin{equation}
\label{eq:rho1rho2}
 \left\{
\begin{aligned}
 \frac{d}{dt} \rho_1(t)= \rho_1(t) \left( H_1+\mathcal{D}_1(t) -\mu_{11} \rho_1(t)-\mu_{12}\rho_2(t) \right),\\
\frac{d}{dt} \rho_2(t)= \rho_2(t) \left( H_2+\mathcal{D}_2(t) -\mu_{21} \rho_1(t)-\mu_{22}\rho_2(t) \right),
\end{aligned}
\right.
\end{equation}
with
\begin{equation*}
   \mathcal{D}_i(t)= \int_{\mathcal{X}}a_i(y)\dfrac{g_i(t,y)}{\rho_i(t)}dy-H_i +\sum_{j=1,2}\left(  \int_{\mathcal{X}}I_{ij}(y)\dfrac{g_j(t,y)}{\rho_j(t)}dy-\mu_{ij}  \right)\rho_j(t) \underset{t \rightarrow +\infty}{\longrightarrow} 0,\quad \text{for } i\in \{1,2\}.
\end{equation*}
Here we have used the fact that, in view of \fer{as:ai} and \fer{hyp:ci}, $\rho_j$ is a positive solution to 
$$
\partial_t\rho_j(t) \leq (a_{\infty}-I_{-}\rho_j(t))\rho_j(t)
$$ 
and hence $\rho_j$ is bounded, for $j\in \{1,2\}$.\\

\noindent
To go further we need the following lemma which is proven below. 

\begin{lemma}
\label{lemme:rho1rho2}
 Let $(\rho_1(t), \rho_2(t))$ be a positive solution to
\begin{equation*}
 \left\{
\begin{aligned}
 \frac{d}{dt} \rho_1(t)= \rho_1(t) \left( H_1+\mathcal{E}_1(t) -\mu_{11} \rho_1(t)-\mu_{12}\rho_2(t) \right),\\
\frac{d}{dt} \rho_2(t)= \rho_2(t) \left( H_2+\mathcal{E}_2(t) -\mu_{21} \rho_1(t)-\mu_{22}\rho_2(t) \right),
\end{aligned}
\right.
\end{equation*}
where $\mathcal{E}_i(t) \underset{t \rightarrow +\infty}{\longrightarrow}0$ for $i\in \{1,2\}$.
\begin{itemize}
 \item If $H_1\leq 0$ and $H_2 \leq 0$, then $ (\rho_1(t), \rho_2(t)) \underset{t \rightarrow +\infty}{\longrightarrow} \left(0, 0\right)$.
\end{itemize}
Also if at least one of the two eigenvalues is positive and
\begin{itemize}
 \item if $H_2 \mu_{11}-H_1 \mu_{21}\leq 0$ and $H_1 \mu_{22}-H_2 \mu_{12}>0$, then $ (\rho_1(t), \rho_2(t)) \underset{t \rightarrow +\infty}{\longrightarrow} \left(\frac{H_1}{\mu_{11}}, 0\right)$,
 \item if $H_2 \mu_{11}-H_1 \mu_{21}> 0$ and $H_1 \mu_{22}-H_2 \mu_{12}\leq 0$, then
 $(\rho_1(t), \rho_2(t)) \underset{t \rightarrow +\infty}{\longrightarrow} \left(0, \frac{H_2}{\mu_{22}}\right)$,
\item if $H_2 \mu_{11}-H_1 \mu_{21}> 0$ and $H_1 \mu_{22}-H_2 \mu_{12}>0$ then
$(\rho_1(t), \rho_2(t)) \underset{t \rightarrow +\infty}{\longrightarrow} (r_1, r_2)$,
where $r_1$ and $r_2$ are given by \fer{H1H2},
\item if $H_2 \mu_{11}-H_1 \mu_{21}< 0$ and $H_1 \mu_{22}-H_2 \mu_{12}=0$, or if $H_2 \mu_{11}-H_1 \mu_{21}= 0$ and $H_1 \mu_{22}-H_2 \mu_{12}< 0$, then $(\rho_1(t), \rho_2(t))$ has a limit which can be
 $\left( \frac{H_1}{\mu_{11}}, 0 \right) \text{ or } \left( 0, \frac{H_2}{\mu_{22}} \right)$, depending on  the initial condition, on the parameters and on the functions $(\mathcal{E}_i)_{i=1,2}$,
\item finally, if $H_2 \mu_{11}-H_1 \mu_{21}< 0$ and $H_1 \mu_{22}-H_2 \mu_{12}< 0$, then $(\rho_1(t), \rho_2(t))$ has a limit which can be one of the three non-zero limits, depending on  the initial condition, on the parameters and on the functions $(\mathcal{E}_i)_{i=1,2}$.
\end{itemize}
\end{lemma}

\noindent
This lemma and expressions \eqref{limit:gi/intgi} are sufficient to prove all the statements  of the theorem:\\

\noindent
(\ref{item:convergence}) One can verify using \fer{as:det} that all possible values of $(H_1,H_2)\in \R^2$ are covered by Lemma \ref{lemme:rho1rho2} and hence, in all cases,  the solution to \eqref{eq:rho1rho2} has a limit when $t$ tends to $+\infty$ for any initial condition.\\

\noindent 
(\ref{item:00}) If $H_1 \leq 0$ and $H_2 \leq 0$, $(\rho_1(t), \rho_2(t))$ tends to $(0,0)$, so for any initial condition, $(g_1(t,.),g_2(t,.))$ tends to $(0,0)$.\\

\noindent
(\ref{item:g10gas}) If $H_1 > 0 , \;  H_2\mu_{11}-H_1\mu_{21} \leq 0 \text{ and } H_1\mu_{22}-H_2\mu_{12} >0$, $(\rho_1(t), \rho_2(t))$ tends to $(\frac{H_1}{\mu_{11}},0)$ and $\int_{\mathcal{X}}\bar{g}_1(x)dx=\frac{H_1}{\mu_{11}}$. Therefore, for any initial condition, $(g_1(t,.),g_2(t,.))$ tends to $(\bar{g}_1,0)$.\\

\noindent
(\ref{item:0g2}) If $H_2 > 0 , \;  H_2\mu_{11}-H_1\mu_{21} > 0 \text{ and } H_1\mu_{22}-H_2\mu_{12} \leq0$,  $(\rho_1(t), \rho_2(t))$ tends to $(0,\frac{H_2}{\mu_{22}})$ and $\int_{\mathcal{X}}\bar{g}_2(x)dx=\frac{H_2}{\mu_{22}}$. Therefore, for any initial condition, $(g_1(t,.),g_2(t,.))$ tends to $(0, \bar{g}_2)$.\\

\noindent
(\ref{item:coexistence}) If $H_1 >0, \; H_2 >0, \; H_2\mu_{11}-H_1\mu_{21} > 0 \text{ and } H_1\mu_{22}-H_2\mu_{12} >0$, $(\rho_1(t), \rho_2(t))$ tends to $(r_1, r_2)$  and we have, from Lemma \ref{lemme:etatstationnaire}, that $r_1=\int_{\mathcal{X}}\hat{g}_1(x)dx$ and $r_2=\int_{\mathcal{X}}\hat{g}_2(x)dx$. It follows that $(g_1(t,.),g_2(t,.))$ tends to $(\hat{g}_1, \hat{g}_2)$ for any initial condition.\\


\begin{proof}[Proof of Lemma \ref{lemme:rho1rho2}]
We split the proof into several cases depending on the values of $H_1$ and $H_2$.\\

\noindent
\textbf{Case 1}: First of all, we will consider that at least  one of the two variables is non-positive. For example, let assume that $H_1 \leq 0$.\\
Let $\varepsilon>0$ and $t_{\varepsilon}>0$ such that for all $t\geq t_{\varepsilon}$, $|\mathcal{E}_1(t)|\leq \mu_{11} \varepsilon$. So for all $t \geq t_{\varepsilon}$, $\partial_t \rho_1(t) \leq \rho_1(t) (\mu_{11}\varepsilon-\mu_{11} \rho_1(t))$. Thanks to the results on the logistic equation, we conclude easily that $ \underset{t \rightarrow +\infty}{\limsup} \rho_1(t) \in [0, \varepsilon]$. As this is true for all $\varepsilon>0$, $\rho_1(t)$ tends toward $0$ when $t$ approaches infinity. Therefore, $\rho_2$ solves
\begin{equation*}
\partial_t \rho_2(t)= \rho_2(t) (H_2+\mathcal{E}'(t)-\mu_{22}\rho_2(t)) \text{ where } \mathcal{E}'(t)=\mathcal{E}_2(t)-\mu_{21}\rho_1(t) \underset{t \rightarrow +\infty}{\longrightarrow} 0.
\end{equation*}
We conclude that $\rho_2$ convergence and evaluate its limit thanks to Lemma \ref{lemme:rho}.\\

\bigskip\noindent
We consider now that $H_1$ and $H_2$ are positive. We will detail only three cases here, the others can be adapted from those three cases.\\
\textbf{Case 2}: Let $H_2 \mu_{11}-H_1 \mu_{21}<0$ and $H_1 \mu_{22}-H_2 \mu_{12}>0$; the case where $H_2 \mu_{11}-H_1\mu_{21}>0$ and $H_1\mu_{22}-H_2\mu_{12}<0$, can be studied following similar arguments.

\bigskip \noindent
Let $\varepsilon>0$ be small enough to satisfy
\begin{equation}
\label{choice:eps}
 \min\left\{ \frac{H_1-\varepsilon}{\mu_{12}}-\frac{H_2+\varepsilon}{\mu_{22}}, \frac{H_1-\varepsilon}{\mu_{11}}-\frac{H_2+\varepsilon}{\mu_{21}}, \frac{H_2-\varepsilon}{\mu_{22}},\frac{H_2-\varepsilon}{\mu_{21}} \right\}>0.
\end{equation}
We split $(\mathbb{R}_+)^2$ into five disjoint sets presented bellow and drawn on figure \ref{fig1}:
\begin{align*}
&D_1=\{ (\rho_1, \rho_2) \in (\mathbb{R}_+^*)^2, \, -H_2+\mu_{21}\rho_1+\mu_{22}\rho_2 \leq -\varepsilon \}\\
&D_2=\{ (\rho_1, \rho_2) \in (\mathbb{R}_+^*)^2,\, -H_1+\mu_{11}\rho_1+\mu_{12}\rho_2 \geq \varepsilon \}\\
&D_3=\{ (\rho_1, \rho_2) \in (\mathbb{R}_+^*)^2,\, -H_2+\mu_{21}\rho_1+\mu_{22}\rho_2 \geq \varepsilon, \; -H_1+\mu_{11}\rho_1+\mu_{12}\rho_2 \leq -\varepsilon \}\\
&D_4=\{ (\rho_1, \rho_2) \in (\mathbb{R}_+^*)^2,\, -H_2+\mu_{21}\rho_1+\mu_{22}\rho_2 \geq -\varepsilon, \; -H_2+\mu_{21}\rho_1+\mu_{22}\rho_2 \leq \varepsilon \}\\
&D_5=\{ (\rho_1, \rho_2) \in (\mathbb{R}_+^*)^2,\, -H_1+\mu_{11}\rho_1+\mu_{12}\rho_2 \geq -\varepsilon, \; -H_1+\mu_{11}\rho_1+\mu_{12}\rho_2 \leq \varepsilon \}.
\end{align*}
\begin{figure}[!ht]
\begin{center}
\includegraphics[scale=0.4]{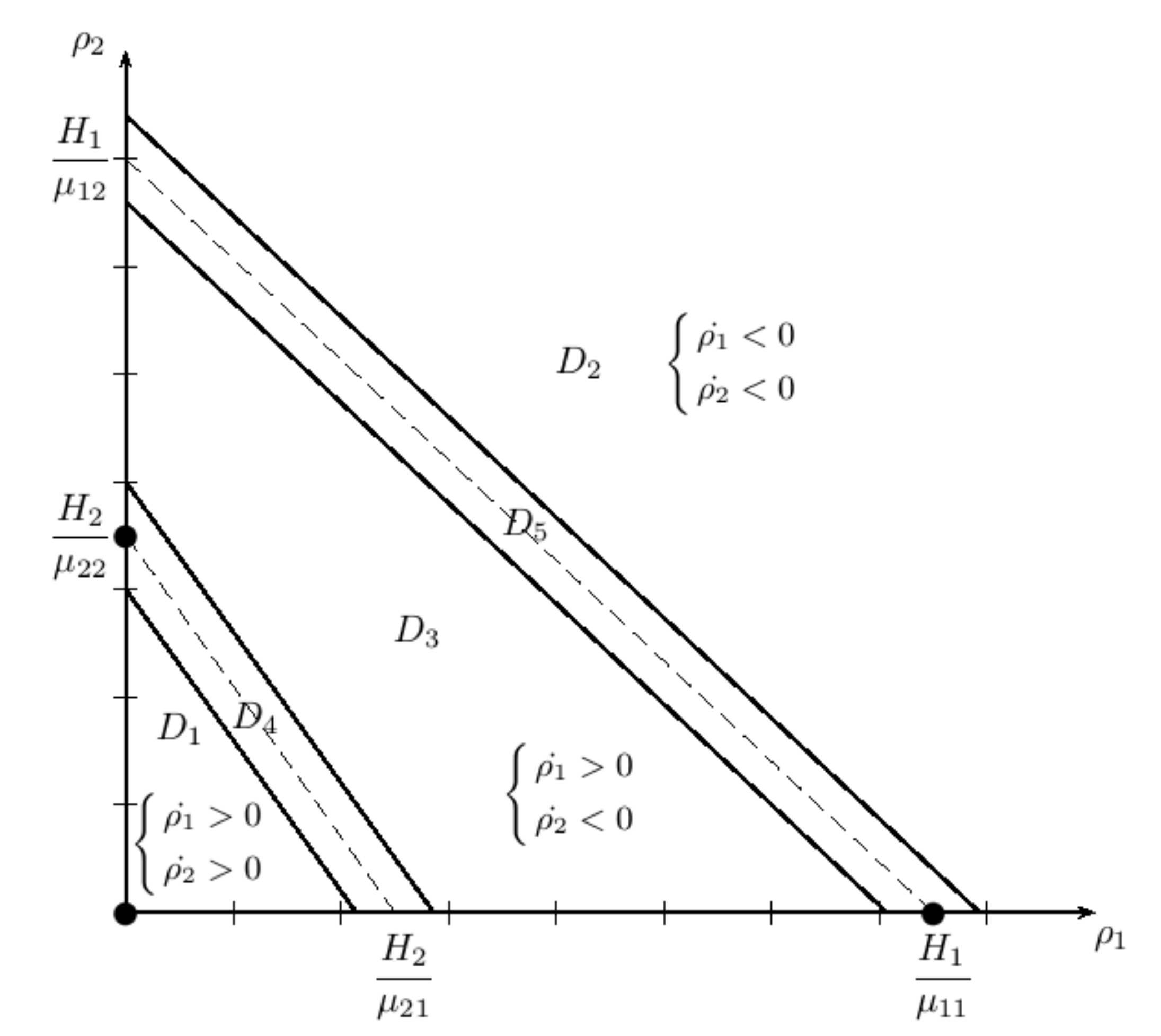}
\caption{\label{fig1} Plan arrangement for case 2, i.e. $H_2 \mu_{11}-H_1 \mu_{21}<0$ and $H_1 \mu_{22}-H_2 \mu_{12}>0$}
\end{center}
\end{figure}
\noindent 
There exists $t_{\varepsilon}>0$ such that for all $t\geq t_{\varepsilon}$, $\max\{|\mathcal{E}_1(t)|, |\mathcal{E}_2(t)|\} \leq \frac{\varepsilon}{2}$. It is then easy to verify that, for $i=1,2$, $\f {d}{dt} \rho_i \geq \f {\varepsilon} {2}\rho_i$ in $D_1$ and $\f {d}{dt}\rho_i\leq -\f {\varepsilon} {2}\rho_i$ in $D_2$. Moreover, $\f {d}{dt} \rho_1 \geq \f {\varepsilon} {2}\rho_1$ and $\f {d}{dt} \rho_2 \leq -\f {\varepsilon} {2}\rho_2$ in $D_3$.\\

\noindent 
As $\rho_1$ satisfies $\f {d}{dt} \rho_1(t) \geq \frac{\varepsilon}{2} \rho_1(t)$ in $D_1$, for all $t \geq t_{\varepsilon}$, if $(\rho_1(\bar t), \rho_2(\bar t))$ belongs to $D_1$ for some $\bar t\geq t_{\varepsilon}$, it will quit this domain after a finite time $t_0$ and reach the set $D'=D_3 \cup D_4 \cup D_5$. Same kind of results holds in $D_2$. Thus after a finite time $t_0 \geq t_{\varepsilon}$, the trajectory of the solution reaches $D'$, moreover it cannot quit this domain according to the signs of derivatives of $\rho_1$ and $\rho_2$ at the boundaries of $D'$.\\
The next step is to study the dynamic in $D'$. Suppose that the trajectory belongs to $D_3 \cup D_4$, it cannot stay in that area for all $t\geq t_0$, so there exists $t_1$ such that $(\rho_1(t_1), \rho_2(t_1)) \in D_5$. We denote the entry point in $D_5$ by $x_1$, drawn in figure \ref{fig2}. According to the derivatives of ${\rho_1}$ and ${\rho_2}$, the trajectory of the solution does not quit the set:
\begin{equation*}
 D_{x_1}=\left\{ (\rho_1,\rho_2)\in (\mathbb{R}_+^*)^2, \rho_1 \geq \rho_1^{x_1}, \rho_2 \leq \rho_2^{x_1} \right\} \cap D',
\end{equation*}
where $(\rho_1^{x_1}, \rho_2^{x_1})$ are the coordinates of $x_1$, this set is represented by the hatched area on the left scheme of figure \ref{fig2}.
\noindent
Moreover, as long as the trajectory stays in $D_3 \cup D_5$, $\rho_2$ satisfies $\partial_t \rho_2(t) \leq -\frac{\varepsilon}{2} \rho_2(t)$. So two cases can happen:
\begin{itemize}
\item[(a)] either $\rho_2(t) \underset{t \rightarrow +\infty}{\longrightarrow}0 $, and thus $\rho_1$ tends to $\frac{H_1}{\mu_{11}}$ from Lemma \ref{lemme:rho},
\item[(b)] or there exists $t_2>t_1$ where the trajectory reaches $D_4$. Let denote $x_2=(\rho_1^{x_1}, \frac{H_2+\varepsilon-\mu_{21}\rho^{x_1}_1}{\mu_{22}})$. As the trajectory stays in $D_{x_1}$, it reaches $D_4\cap D_{x_2}$, where $D_{x_2}=\{ (\rho_1,\rho_2)\in (\mathbb{R}_+^*)^2, \rho_1\geq\rho_1^{x_2}, \rho_2\leq\rho_2^{x_2} \} \cap D'$ . Moreover, for all $t\geq t_2 $, the trajectory stays in $D_{x_2}$ (see the hatched area on the right scheme of figure \ref{fig2}).\\
Iterating the previous step, we construct a decreasing sequence of areas $(D_{x_n})_{n\geq 0}$ which will be necesseraly finite. Indeed, the choice of $\varepsilon$ \eqref{choice:eps} implies that there exists $m \in \mathbb{N}$ such that $D_{x_{2m}} \cap D_4=\varnothing$. Then we conclude as in the case (a) above.
\end{itemize}

\begin{figure}[!ht]
\begin{center}
\includegraphics[scale=0.25]{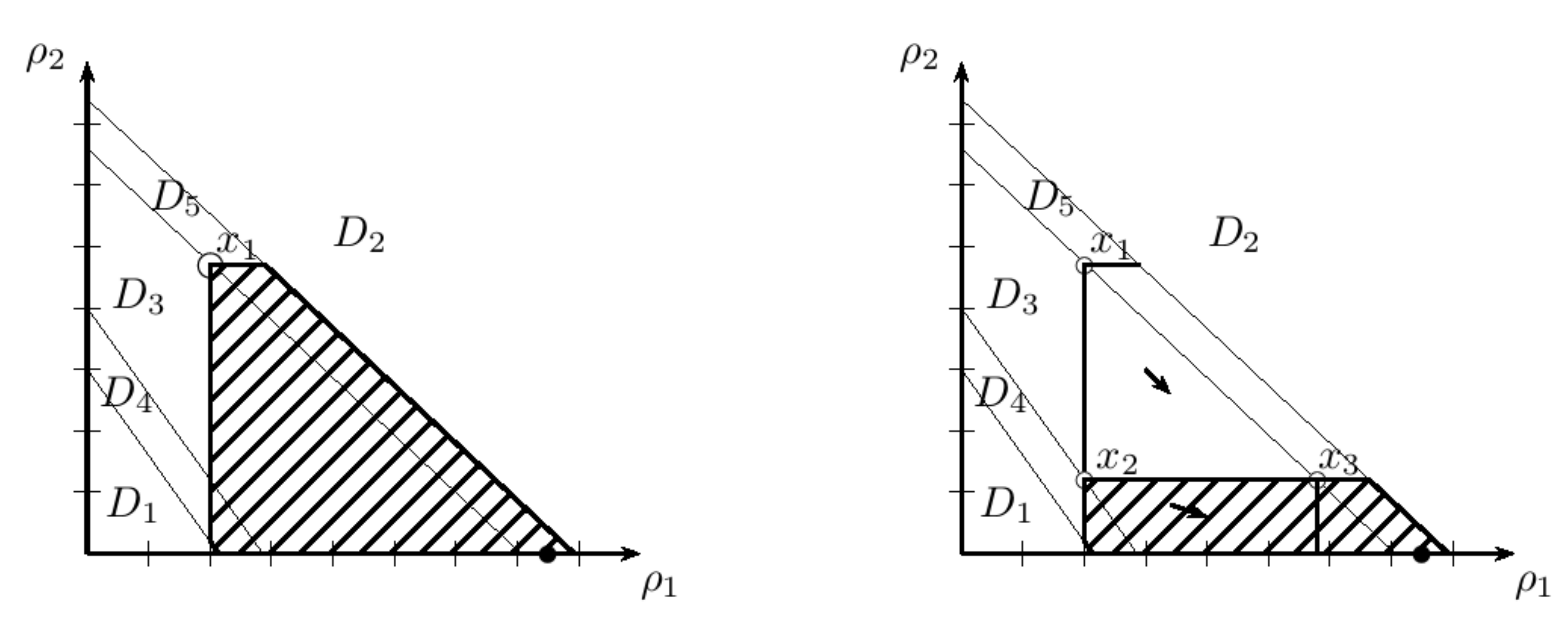}
\caption{\label{fig2} Dynamic for the case 2: $H_2 \mu_{11}-H_1 \mu_{21}<0$ and $H_1 \mu_{22}-H_2 \mu_{12}>0$}
\end{center}
\end{figure}

\bigskip
\noindent
The next case is quite similar except for the end of the proof.\\\
\textbf{Case 3}: Let $H_2 \mu_{11}-H_1 \mu_{21}=0$ and $H_1 \mu_{22}-H_2 \mu_{12}>0$;  
 the case where $H_2 \mu_{11}-H_1\mu_{21}>0$ and $H_1\mu_{22}-H_2\mu_{12}=0$ and the one where $H_2 \mu_{11}-H_1\mu_{21}>0$ and $H_1\mu_{22}-H_2\mu_{12}>0$ can be proven using same kind of arguments.

\bigskip \noindent Let $k\geq 1$ and $\varepsilon>0$ be such that
\begin{equation*}
 \max \left\{\frac{\mu_{11}}{\mu_{21}},\frac{\mu_{21}}{\mu_{11}} \right\} <k \quad \text{ and } \quad \min\left\{ \frac{H_1-\varepsilon}{\mu_{21}}-\frac{H_2+\varepsilon}{\mu_{22}}, \frac{H_2-k\varepsilon}{\mu_{22}},\frac{H_2-k\varepsilon}{\mu_{21}} \right\}>0.
\end{equation*}
We divide the plan $(\mathbb{R}_+)^2$ as presented in figure \ref{cas1'}, where $D_1$, $D_2$, $D_3$, $D_4$, $D_5$ and $D'$ are defined as follows
\begin{align*}
&D_1=\{ (\rho_1, \rho_2) \in (\mathbb{R}_+^*)^2, \, -H_2+\mu_{21}\rho_1+\mu_{22}\rho_2 \leq -k\varepsilon \}\\
&D_2=\{ (\rho_1, \rho_2) \in (\mathbb{R}_+^*)^2,\, -H_1+\mu_{11}\rho_1+\mu_{12}\rho_2 \geq k\varepsilon \}\\
&D_3=\{ (\rho_1, \rho_2) \in (\mathbb{R}_+^*)^2,\, -H_2+\mu_{21}\rho_1+\mu_{22}\rho_2 \geq \varepsilon, \; -H_1+\mu_{11}\rho_1+\mu_{12}\rho_2 \leq -\varepsilon \}\\
&D_4=\{ (\rho_1, \rho_2) \in (\mathbb{R}_+^*)^2,\, -H_2+\mu_{21}\rho_1+\mu_{22}\rho_2 \geq -k\varepsilon, \; -H_2+\mu_{21}\rho_1+\mu_{22}\rho_2 \leq \varepsilon \}\\
&D_5=\{ (\rho_1, \rho_2) \in (\mathbb{R}_+^*)^2,\, -H_1+\mu_{11}\rho_1+\mu_{12}\rho_2 \geq -\varepsilon, \; -H_1+\mu_{11}\rho_1+\mu_{12}\rho_2 \leq k\varepsilon \}\\
& D'=D_3 \cup D_4 \cup D_5.
\end{align*}
\begin{figure}[!ht]
\begin{center}
\includegraphics[scale=0.4]{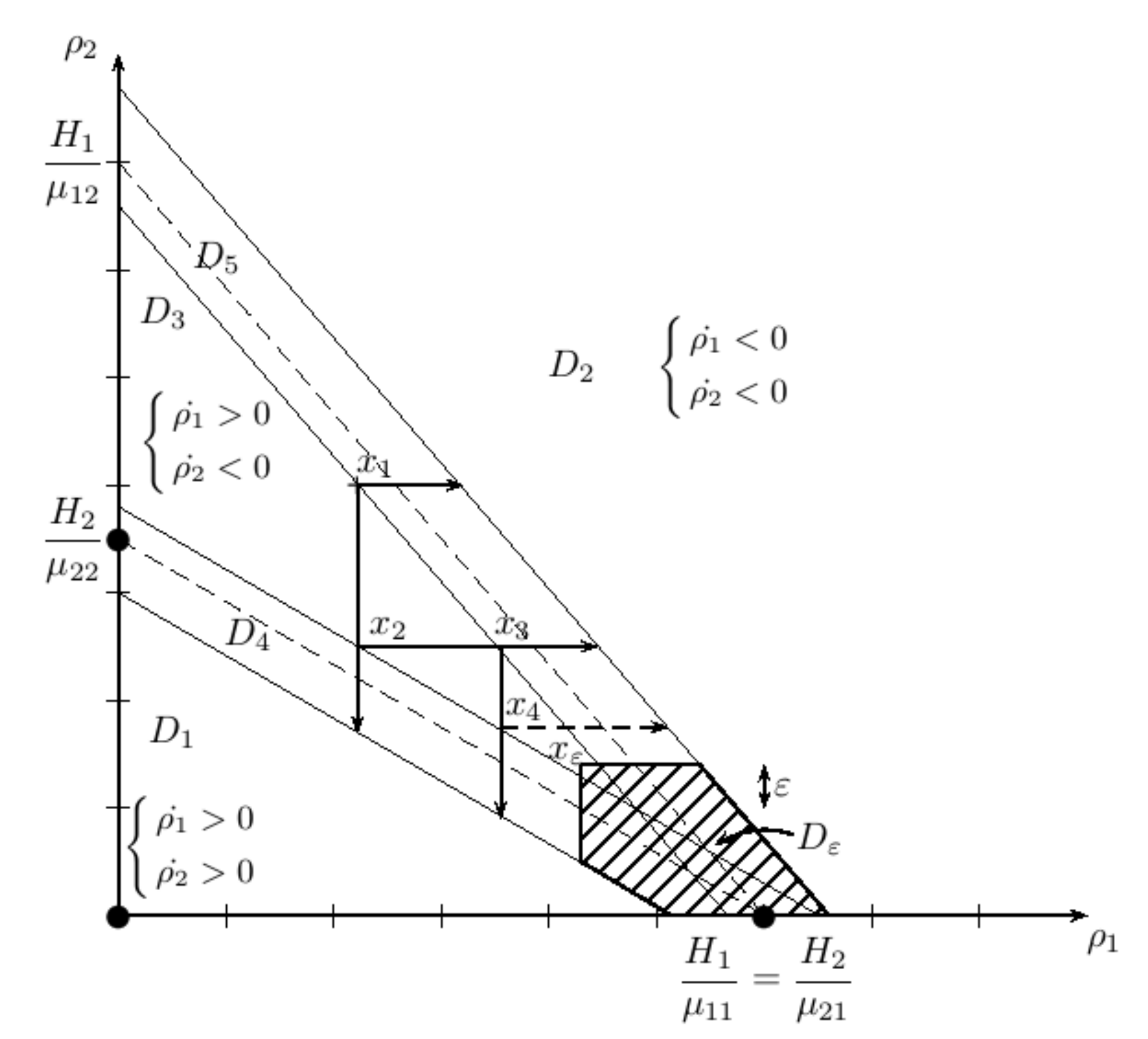}
\caption{\label{cas1'} Arrangement for the case 3, i.e. $H_2 \mu_{11}-H_1 \mu_{21}=0$ et $H_1 \mu_{22}-H_2 \mu_{12}>0$}
\end{center}
\end{figure}

\noindent
The constant $k$ is chosen such that $D_1\cap D_5=\varnothing$ and $D_2 \cap D_4=\varnothing$ and $t_{\varepsilon}$ is defined as before. There exists $t_0\geq t_{\varepsilon}$ such that for all $t\geq t_0$, the trajectory is belonging to $D'$. Then we construct a sequence of sets $(D_{x_n})_{n\geq 1}$ as before, but this time, this sequence can be infinite. So let $D_{\varepsilon}$ be the set
\begin{multline*}
 D_{\varepsilon}=D' \cap \bigg\{ (\rho_1, \rho_2)\in (\mathbb{R}^+)^2,  \rho_2\leq \varepsilon \left( \dfrac{\mu_{11}+\mu_{21}}{\mu_{11}\mu_{22}-\mu_{12}\mu_{21}} +1\right)\\
 \text{ and } \rho_1 \geq \dfrac{H_1}{\mu_{11}} -\varepsilon\left( \dfrac{\mu_{22}+\mu_{12}}{\mu_{11}\mu_{22}-\mu_{12}\mu_{21}} +1 \right)\bigg\}.
\end{multline*}
There exists $n$ such that $D_{x_n}$ is included in $D_{\varepsilon}$, i.e. the trajectory is belonging to $D_{\varepsilon}$ after a finite time. As this is true for all $\varepsilon>0$, $(\rho_1(t), \rho_2(t)) \underset{t \rightarrow +\infty}{\longrightarrow} \left(\frac{H_1}{\mu_{11}}, 0\right)$.\\

\bigskip \noindent
The last case that we detail is a case where several limits are possible.\\
\textbf{Case 4}: Let $H_2 \mu_{11}-H_1 \mu_{21}<0$ and $H_1 \mu_{22}-H_2 \mu_{12}=0$; 
we can deal with the case $H_2 \mu_{11}-H_1\mu_{21}=0$ and $H_1\mu_{22}-H_2\mu_{12}<0$ and the one with $H_2 \mu_{11}-H_1\mu_{21}<0$ and $H_1\mu_{22}-H_2\mu_{12}<0$ thanks to similar arguments.

 \medskip \noindent Let $k\geq 1$ and $\varepsilon>0$ such that
$$\max \left\{\frac{\mu_{22}}{\mu_{12}},\frac{\mu_{12}}{\mu_{22}} \right\} <k \quad \text{ and } \quad \min\left\{ \frac{H_1-\varepsilon}{\mu_{11}}-\frac{H_2+\varepsilon}{\mu_{21}}, \frac{H_2-k\varepsilon}{\mu_{22}},\frac{H_2-k\varepsilon}{\mu_{21}} \right\}>0.$$
We divide the plan $(\mathbb{R}_+)^2$ as presented in figure \ref{cas1''}, where $D_3$, $D_4$, $D_5$, $D'$ are defined as in the case 3 and $D_{\varepsilon}$ is defined as follows
\begin{align*}
 D_{\varepsilon}=D' \setminus \bigg\{ (\rho_1,\rho_2)\in (\mathbb{R}^+)^2, \rho_1> \varepsilon  \dfrac{\mu_{21}+\mu_{22}}{\mu_{12}\mu_{21}-\mu_{11}\mu_{22}} \, \text{ and } \, \rho_2< \dfrac{H_2}{\mu_{22}}-\varepsilon \dfrac{\mu_{12}+\mu_{11}}{\mu_{12}\mu_{21}-\mu_{11}\mu_{22}} \bigg\}
 \end{align*}

\begin{figure}[!ht]
\begin{center}

\includegraphics[scale=0.4]{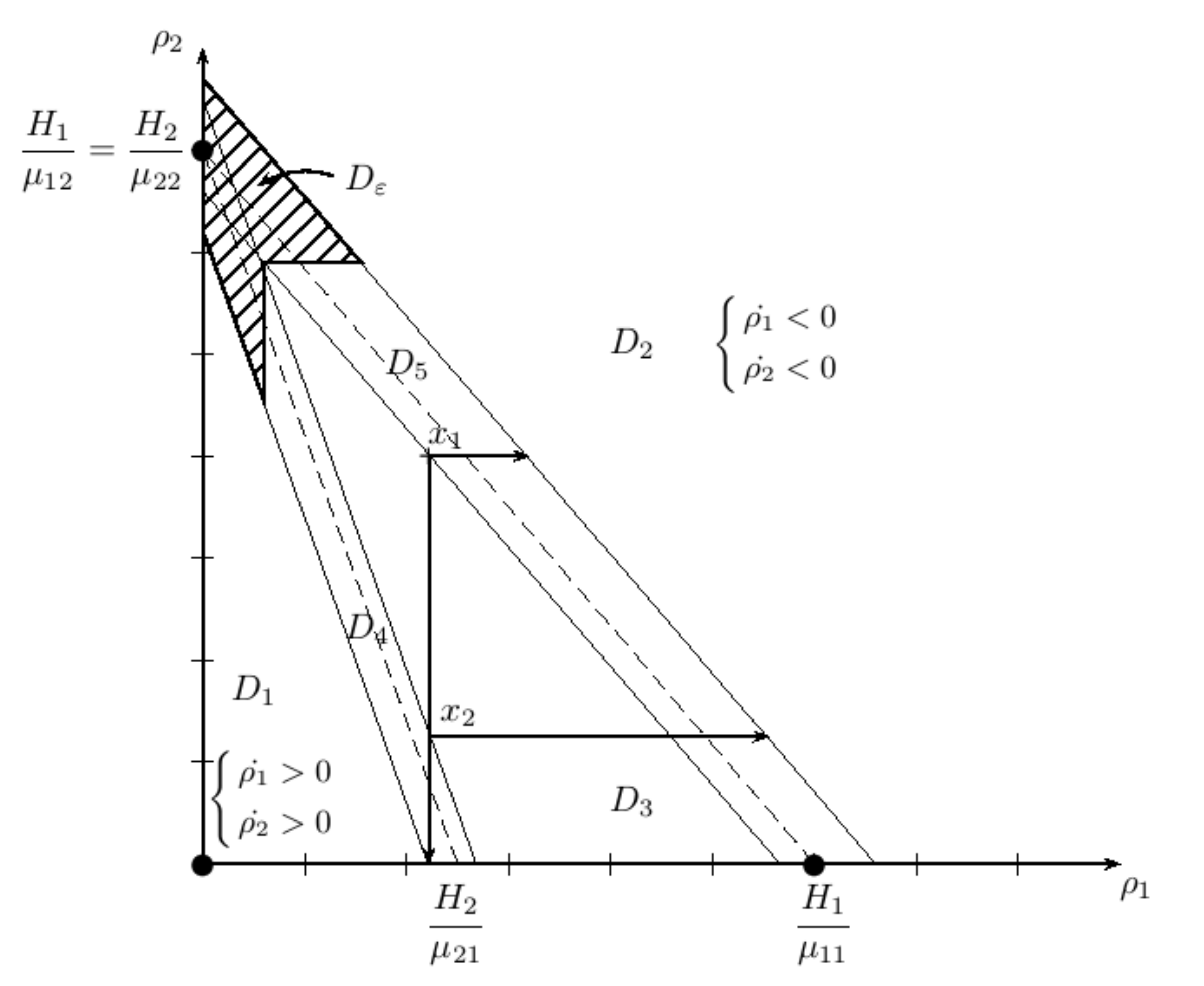}
\caption{\label{cas1''} Arrangement for the case 4, i.e. $H_2 \mu_{11}-H_1 \mu_{21}<0$ et $H_1 \mu_{22}-H_2 \mu_{12}=0$}
\end{center}
\end{figure}
\noindent
As before, we find $t_0\geq t_{\varepsilon}$ such that for all $t\geq t_0$, the trajectory of the solution belongs to $D'$. Then there exist two possibilities.
\begin{itemize}
\item Either for all $\varepsilon>0$, there exists $\tau_{\varepsilon} > t_{\varepsilon}$ such that for all $t \geq \tau_{\varepsilon}$ the trajectory belongs to $D_{\varepsilon}$, that is, $ (\rho_1(t),\rho_2(t)) \underset{t \rightarrow +\infty}{\longrightarrow} \left( 0, \frac{H_2}{\mu_{22}} \right).$
\item Or there exists $\varepsilon>0$ and $\tau_{\varepsilon}>t_{\varepsilon}$ when the trajectory is belonging to $D' \setminus D_{\varepsilon}$. Using same kind of arguments as before, we obtain that the trajectory won't quit this set for all $t \geq \tau_{\varepsilon}$, and we construct a sequence of decreasing sets to conclude that $ (\rho_1(t), \rho_2(t)) \underset{t \rightarrow +\infty}{\longrightarrow} \left(\frac{H_1}{\mu_{11}}, 0\right)$.\\
\end{itemize}

\end{proof}

\bigskip

\subsection{The proof of Theorem \ref{theo:special}}
\noindent
In this section we prove Theorem \ref{theo:special}.\\

\noindent (\ref{item:special1}) Let us deal with the first case where $H_1 > 0, \; H_2 >0, \; H_2\mu_{11}-H_1\mu_{21} <0,  \text{ and } H_1\mu_{22}-H_2\mu_{12}<0$. Thanks to the last statement of Lemma \ref{lemme:rho1rho2}, we already know that any solution tends towards one of the non-trivial steady states. We precise the stability of each state.\\

(\ref{item:special1}a) This point is devoted to show the asymptotic stability of $(\bar{g}_1,0)$ if $H_1>0$ and $H_2 \mu_{11}-H_1 \mu_{21}<0$. Using symmetric arguments, it then can be shown that if $H_2>0$ and $H_1\mu_{22}-H_2\mu_{12}<0$, $(0, \bar{g}_2)$ is stable.\\

\noindent
Precisely, we show that if the positive initial condition $(g_1(0,.),g_2(0,.))$ satisfies the following condition: there exist $C_1>0$ and $C_2>0$ such that
   \begin{eqnarray}
\label{hyp:g1}
    \max_{i=1,2} \{ \|g_1(0,.)-\bar{g}_1\|_{L^2} \|I_{i1}\|_{L^2} \} &\leq& {C_1},\\
\label{hyp:g2}
    \|g_2(0,.)\|_{L^2} {\|I_{12}\|_{L^2}} &\leq& {C_2},
   \end{eqnarray}
   where the above constants satisfy the following compatibility conditions
   \beq
    \label{hyp1C}
    {C} =2 \, \left((C_1+C_2) \cdot \max \left\{ 1, \dfrac{\mu_{21}}{\mu_{11}} \right\} +  C_1 \right) <\min \left\{ H_1-\lambda_2^1, \dfrac{\mu_{21}}{\mu_{11}} H_1 -H_2 \right\},
   \eeq
   then the solution to the equation \eqref{eq:systeme} tends to the steady state $(\bar{g}_1,0)$.\\
   
\noindent   
Let us express $g_1$ in the basis $(A_k^1, k \in \mathbb{N}^*)$, $g_1(t,x)=\int_{\mathcal{X}}\bar{g}_1(x)dx A_1^1(x) + \sum_{k=1}^{\infty} \alpha_k(t) A_k^1(x)$, $\forall x\in \mathcal{X}$, and denote $\kappa(t)=\int_{\mathcal{X}}\bar{g}_1(x)dx+ \alpha_1(t)$, for all $t\in \mathbb{R}^+$.\\

\noindent
From \eqref{eq:systeme} and the representation of $g_1$ and $\partial_t g_1$ with respect to the basis $(A_k^1, k \in \mathbb{N}^*)$, we find the following dynamical system
\begin{equation}
\label{eq:alpharho}
   \left\{
    \begin{aligned} 
       &\dfrac{d}{dt}\alpha_k(t)= \alpha_k(t) \left(\lambda_k^1-H_1-  \sum_{\ell=1}^{\infty} \alpha_{\ell}(t) \int_{\mathcal{X}}({I_{11}}A_{\ell}^1) -\int_{\mathcal{X}}({I_{12}}g_2(t,.)) \right), \; \forall k \geq 2,\\
       &\dfrac{d}{dt}\kappa(t)=\kappa(t) \left( H_1-\sum_{\ell=2}^{\infty} \alpha_{\ell}(t)\int_{\mathcal{X}}({I_{11}}A_{\ell}^1)  - \int_{\mathcal{X}}({I_{12}}g_2(t,.)) -\mu_{11} \kappa(t) \right).
    \end{aligned}
   \right.
\end{equation} 
Here,  we have used the fact that since, from Lemma \ref{spec}, $L_N=\int_{\mathcal{X}}\bar{g}_1dx A_1^1+\sum_{k=1}^N \al_k A_k^1$ tends to $g_1$ in $L^2$ as $N\to \infty$, and since the domain $\CX$ is bounded,  $L_N$ tends to $g_1$ in $L^1$.\\

\noindent We will show that for all $t\geq 0$, 
\begin{equation}
\label{eq:lowerbound}
\min_{i=1,2}  \left\{  \sum_{k=1}^{\infty} \alpha_{k}(t)\int_{\mathcal{X}}({I_{i1}(y)}A_{k}^1(y))dy \right\} \geq  -C.
\end{equation}
We notice that this is true for $t=0$ thanks to \eqref{hyp:g1}: for $i\in \{1,2\}$,
\begin{equation*}
 \begin{aligned}
  \left\vert \sum_{k=1}^{\infty} \alpha_{k}(0)\int_{\mathcal{X}}({I_{i1}}A_{k}^1) \right\vert & \leq \left(\sum_{k=1}^{\infty} \alpha_{k}(0)^2 \int_{\mathcal{X}} |A_{k}^1|^2 \right)^{1/2} \left(\sum_{k=1}^{\infty} \dfrac{  (\int_{\mathcal{X}} I_{i1} A_{k}^1)^2}{\int_{\mathcal{X}} |A_{k}^1|^2}\right)^{1/2}\\
 & \leq \|g_1(0,.)-\bar{g}_1\|_{L^2} \|I_{i1}\|_{L^2} \\
 &\leq C_1< C.
 \end{aligned}
\end{equation*}
The second line is justified using the representation of $I_{i1}$ with respect to the orthonormal basis $\{A_k^1/\|A_k^1\|_{L^2}\}_{k\geq 1}$ which is $\{(\int_{\mathcal{X}}I_{i1}A_k^1)/\|A_k^1\|_{L^2} \}_{k\geq 1}$.\\

\noindent
We prove the result \eqref{eq:lowerbound} by contradiction. Denote 
$$ t_0=\inf \left\{ t>0 \Big\vert \min_{i=1,2} \Big\{\sum_{k=1}^{\infty} \alpha_k(t)\int_{\mathcal{X}}({I_{i1}(y)}A_k^1(y))dy\Big\}\leq -C  \right\}
$$
and suppose that $t_0$ is finite. Thus, for $i\in \{1,2\}$,
\begin{equation}
\label{assumption}
 \forall t \leq t_0, \quad \sum_{k=1}^{\infty} \alpha_k(t)\int_{\mathcal{X}}({I_{i1}(y)}A_k^1(y))dy \geq -C.
\end{equation}
In fact, from \eqref{assumption}, we will find a lower bound greater than $-C$ for any $t\leq t_0$ which is a contradiction with the fact that $t_0$ is finite.\\
First, let us deal with the expression $\sum_{k=2}^{\infty} \alpha_k(t)\int_{\mathcal{X}}({I_{i1}}A_k^1)dx$. We multiply the first equation of \eqref{eq:alpharho} by $\alpha_k$. Then, using the positivity of $g_2$, the assumption \eqref{hyp1C} and Gronwall's lemma, we get for all $t \leq t_0$, $\alpha_k(t)^2  \leq \alpha_k(0)^2 e^{2(\lambda_k^1-H_1+C)t} \leq \alpha_k(0)^2$. Thus, for all $t \leq t_0$, for $i\in \{1,2\}$,
\begin{equation}
\label{eq:sumalpha}
\begin{aligned}
\left\vert \sum_{k=2}^{\infty} \alpha_k(t)\int_{\mathcal{X}}({I_{i1}}A_k^1) \right\vert & \leq    \left(\sum_{k=2}^{\infty} \alpha_k(t)^2 \int_{\mathcal{X}} |A_k^1|^2 \right)^{1/2} \left(\sum_{k=2}^{\infty} \dfrac{  (\int_{\mathcal{X}} I_{i1} A_k^1 )^2}{\int_{\mathcal{X}} |A_k^1|^2}\right)^{1/2}\\
& \leq  \left(\sum_{k=2}^{\infty} \alpha_k(0)^2 \int_{\mathcal{X}} |A_k^1|^2 \right)^{1/2} \left(\sum_{k=2}^{\infty} \dfrac{  (\int_{\mathcal{X}} I_{i1} A_k^1 )^2}{\int_{\mathcal{X}} |A_k^1|^2}\right)^{1/2}\\
& \leq  \|g_1(0,.)-\bar{g}_1\|_{L^2} \| I_{i1} \|_{L^2} <   C_1.
 \end{aligned}
\end{equation}
Then, in view of finding a lower bound to $\kappa$, we are concerned with $ \int_{\mathcal{X}}{I_{12}}g_2(t,.) dx$. We multiply the second equation in \eqref{eq:systeme} by $g_2$ and integrate it over $\mathcal{X}$:
\begin{equation}
 \frac{1}{2} \frac{d}{dt} \|g_2(t,.)\|_{L^2}^2 {\leq} \left(H_2-\mu_{21}\frac{H_1}{\mu_{11}}-\sum_{k=1}^{\infty} \alpha_k(t)\int_{\mathcal{X}}({I_{21}}A_k^1)-\int_{\mathcal{X}}I_{22}(y)g_2(t,y)dy \right) \cdot \|g_2(t,.)\|_{L^2}^2
\end{equation}
From the assumption \eqref{assumption} for $i=2$, the positivity of $g_2$ and the Gronwall's lemma, we get that for all $t \leq t_0$,
\begin{equation}
\label{eq:intg2}
     \|g_2(t,.)\|_{L^2}^2  \leq \exp \left\{2\left(H_2-\mu_{21}\dfrac{H_1}{\mu_{11}}+C \right)t \right\} \cdot \|g_2(0,.)\|_{L^2}^2 \leq \|g_2(0,.)\|_{L^2}^2.
\end{equation}
That is, with the assumption \eqref{hyp:g2}, for all $t \leq t_0$, $0 \leq \int_{\mathcal{X}}({I_{12}}g_2(t,.))  \leq \|I_{12}\|_{L^2} \|g_2(t,.)\|_{L^2} \leq C_2$.\\
 We use this inequality and \eqref{eq:sumalpha} to show that $\kappa$ satisfies, for all $t \leq t_0$,
 $$
 \kappa(t) \left( H_1-C_1 - C_2 -\mu_{11} \kappa(t) \right)\leq \partial_t\kappa(t).
 $$
Moreover, from \eqref{hyp:g1}, $  \vert \kappa(0)-\frac{H_1}{\mu_{11}} \vert \cdot \mu_{11} \leq  {C_1}$. Using classical results on logistic equation, we deduce the following lower bound
\begin{equation}
\label{eq:rho1}
  \dfrac{-C_1-C_2}{\mu_{11}}  \leq \kappa(t)-\dfrac{H_1}{\mu_{11}} =\alpha_1(t), \; \forall t \leq t_0.
\end{equation}
Finally, we conclude with inequalities \eqref{eq:sumalpha}, \eqref{eq:rho1}, assumption \eqref{hyp1C} and definitions of $\mu_{i1}$: for all $t \leq t_0$, for $i\in \{1,2\}$
\begin{equation*}
 \sum_{k=1}^{\infty} \alpha_k(t) \int_{\mathcal{X}}({I_{i1}}A_k^1) \geq \alpha_1(t) \mu_{i1} +\sum_{k=2}^{\infty} \alpha_k(t) \int_{\mathcal{X}}({I_{i1}}A_k^1) \geq -(C_1+C_2) \dfrac{\mu_{i1}}{\mu_{11}}-C_1 \geq   -\dfrac{C}{2}.
\end{equation*}
This is the contradiction that we wanted to reach, thus, $t_0=+\infty$.\\

\noindent
Moreover, Theorem \ref{theo:main} guarantees the existence of a limit for $(g_1(t,.),g_2(t,.))$. Let us identify that limit. On the one hand, we note that \eqref{eq:intg2} holds for all $t\geq 0$, since $t_0=+\infty$, and hence $ \int_{\mathcal{X}} |g_2(t,x)|^2dx $ tends to $0$, as $t$ approaches infinity.\\ 
On the other hand, thanks to the equation \eqref{eq:rho1}, which holds for all $t \geq 0$, and \eqref{hyp1C},
\begin{eqnarray*}
 \int_{\mathcal{X}} |g_1(t,x)|^2 dx  &\geq& \sum_{k=2}^{\infty} \alpha_k(t)^2 \int_{\mathcal{X}} |A_k^1|^2 + \kappa(t)^2 \int_{\mathcal{X}} |A_1^1|^2\\
 & \geq & 0+\left(\dfrac{H_1-C_1-C_2}{\mu_{11}}\right)^2 \int_{\mathcal{X}} |A_1^1|^2 > 0.
\end{eqnarray*}
Thus, the limit of $\|g_1(t,.)\|_{L^2}$ is positive. The limit of the solution is hence the steady state $(\bar{g}_1,0)$.\\

(\ref{item:special1}b) Here, we show that if $H_2 \mu_{11}-H_1 \mu_{21}<0$ and $H_1 \mu_{22}-H_2 \mu_{12}<0$, the steady state $(\hat{g}_1, \hat{g}_2)$ is unstable, precisely in any neighborhood of $(\hat{g}_1, \hat{g}_2)$, there exists a solution to \eqref{eq:systeme} that does not tend towards $(\hat{g}_1, \hat{g}_2)$, but there also exist some solutions that tend towards it.\\
We use a solution to \eqref{eq:systeme} with an initial condition which belongs to the subspace $vect (\hat{g}_1) \times vect(\hat{g}_2)$. Let us notice, using the form of the equation \eqref{eq:systeme} satisfied by $(g_1,g_2)$, that, if the initial condition belongs to a subspace $vect(A_k^1)_{k\in K} \times vect(A_{\ell}^2)_{\ell\in L}$ with $K$ and $L$ subsets of $\mathbb{N}$, then for all $t\geq 0$, the solution $(g_1(t,.), g_2(t,.))$ belongs to that subspace. Thus, for all $t\geq 0$, $g_1(t,x)=\alpha(t) \hat{g}_1(x)$ and $g_2(t,x)=\beta(t) \hat{g}_2(x)$. We get the following system
\begin{equation*}
 \left\{
	\begin{aligned}
	 &\dfrac{d}{dt} \alpha(t) = \alpha(t) \left( H_1- \mu_{11} \Big( \int_{\mathcal{X}} \hat{g}_1 \Big) \alpha(t) - \mu_{12}\Big( \int_{\mathcal{X}} \hat{g}_2 \Big)\beta(t) \right)\\
	 &\dfrac{d}{dt} \beta(t) = \beta(t) \left( H_2- \mu_{21} \Big(\int_{\mathcal{X}} \hat{g}_1\Big) \alpha(t) - \mu_{22} \Big(\int_{\mathcal{X}} \hat{g}_2\Big) \beta(t) \right).
 	\end{aligned}
\right.
\end{equation*}
We first notice that $(1,1)$ is obviously a steady state here. Moreover, the determinant of the Jacobian matrix of the linearized dynamical system at point $(1,1)$ is $(\mu_{11}\mu_{22}-\mu_{12}\mu_{21}) \int_{\mathcal{X}}\hat{g}_1 \int_{\mathcal{X}} \hat{g}_2<0 $. So the linearized system around $(1,1)$ is hyperbolic. From Hartman-Grobman Theorem (see part 9.3 in \cite{Teschl}) concerning the linearized system, we can conclude that not only $(\hat{g}_1, \hat{g}_2)$ is not stable, but also some solutions tend towards it.\\

\noindent
(\ref{item:special2}) It remains to deal with the last uncertain case : $H_1>0$, $H_2>0$, $H_2 \mu_{11}-H_1 \mu_{21} <0$ and $H_1\mu_{22}-H_2 \mu_{12}=0$. The case (\ref{item:special3}) where $H_2 \mu_{11}-H_1 \mu_{21} =0$ and $H_1\mu_{22}-H_2 \mu_{12}<0$ can be studied following similar arguments.\\

\noindent Thanks to point (\ref{item:special1}a), we already know that the steady state $(\bar{g}_1,0)$ is stable. We prove that the steady state $(0, \bar{g}_2)$ is unstable; more precisely, we show that in any neighborhood of $(0, \bar{g}_2)$, there exists a solution to \eqref{eq:systeme} that does not tend to $(0, \bar{g}_2)$  as $t\to \infty$, but there also exist some solutions that tend towards it. \\

To prove that this steady state is unstable, we only consider solutions in the form of $(g_1,g_2)=(\al(t)A_1^1,\beta(t) A_1^2)$. The dynamics are then given by the following Lotka-Volterra system
\begin{equation*}
 \left\{
	\begin{aligned}
	 &\dfrac{d}{dt} \alpha(t) = \alpha(t) \left( H_1- \mu_{11}  \alpha(t) - \mu_{12}\beta(t) \right),\\
	 &\dfrac{d}{dt} \beta(t) = \beta(t) \left( H_2- \mu_{21} \alpha(t) - \mu_{22}  \beta(t) \right).
 	\end{aligned}
\right.
\end{equation*}
In view of the conditions on the parameters in this case, any solution with $\al(0)>0$, converges to $(\f{H_1}{\mu_{11}},0)$ (see \cite{champagnat-these}, p.186 Theorem 1(c)) and thus $(\al(t)A_1^1,\beta(t) A_1^2)$ converges to $(\bar g_1,0)$. Since one can choose $\al(0)$ and $\beta(0)$ such that $(\al(0)A_1^1,\beta(0) A_1^2)$ is arbitrarily close to $(0,\bar g_2)$, we obtain that this point is unstable.\\

\noindent
 Finally, to find a solution that tends towards the steady state, consider the initial condition $(g_1^0,g_2^0)=(0, g_2^0)$, with $g_2^0\in L^2(\CX)$ any nonnegative and non-trivial function. Then, since for all $t\geq 0$ we have $g_1(t,x)=0$, it follows from Theorem \ref{theo:convmono} that, as $t\to \infty$, $g_2(t,x)\to \bar g_2(x)$ in $L^\infty$.

\medskip
\noindent
\textit{e-mail adress}: helene.leman@polytechnique.edu\\
\textit{e-mail adress}: sylvie.meleard@polytechnique.edu\\
\textit{e-mail adress}: Sepideh.Mirrahimi@math.univ-toulouse.fr

\end{document}